\documentclass{amsart}

\usepackage{helvet, color}

\usepackage{amscd,amsmath,amsxtra,amsthm,amssymb,stmaryrd,xr,mathrsfs,mathtools}
\usepackage[all]{xy}

\newtheorem{theorem}{Theorem}[section]
\newtheorem{lemma}[theorem]{Lemma}

\newtheorem{proposition}[theorem]{Proposition}
\newtheorem{corollary}[theorem]{Corollary}
\newtheorem{definition}[theorem]{Definition}

\newtheorem{remark}[theorem]{Remark}

\numberwithin{equation}{section}

\newcommand{\Y}{\mathcal{Y}}
\newcommand{\Tr}{\operatorname{Tr}}
\newcommand{\Gal}{\operatorname{Gal}}
\newcommand{\Fil}{\operatorname{Fil}}
\newcommand{\cor}{\operatorname{cor}}
\newcommand{\DD}{\mathbb{D}}
\newcommand{\BB}{\mathbb{B}}
\newcommand{\NN}{\mathbb{N}}
\newcommand{\QQ}{\mathbb{Q}}
\newcommand{\Qp}{\mathbb{Q}_p}
\newcommand{\Zp}{\mathbb{Z}_p}
\newcommand{\ZZ}{\mathbb{Z}}
\newcommand{\HH}{\mathbb{H}}
\renewcommand{\AA}{\mathbb{A}}

\newcommand{\sH}{\mathscr{H}}

\DeclareMathOperator{\ord}{ord}
\newcommand{\cH}{\mathcal{H}}
\newcommand{\cE}{\mathcal{E}}
\newcommand{\e}{s}
\newcommand{\lb}{[\![}
\newcommand{\rb}{]\!]}

\newcommand{\fM}{\mathfrak{M}}

\newcommand{\vp}{\varphi}
\newcommand{\cL}{\mathcal{L}}

\newcommand{\cO}{\mathcal{O}}
\newcommand{\Iw}{\mathrm{Iw}}
\newcommand{\HIw}{H^1_{\mathrm{Iw}}}

\newcommand{\T}{\mathcal{T}}
\newcommand{\fv}{\mathfrak{v}}
\newcommand{\n}{\mathfrak{n}}
\newcommand{\Brig}{\BB_{\mathrm{rig}, \Qp}^+}

\newcommand{\AQp}{\AA_{\Qp}^+}
\DeclareMathOperator{\image}{Im}
\newcommand{\Dcris}{\mathbb{D}_{\rm cris}}
\newcommand{\Mlog}{M_{\log}}
\newcommand{\Iarith}{\mathbb{I}_{\mathrm{arith}}}

\newcommand{\Qpn}{\QQ_{p}(\mu_{p^n})}

\newcommand{\Sel}{\mathrm{Sel}}

\newcommand{\Spec}{\mathrm{Spec}\ }
\newcommand{\len}{\mathrm{len}}
\DeclareMathOperator{\coker}{coker}

\newcommand{\Tw}{\mathrm{Tw}}
\renewcommand{\Col}{\mathrm{Col}}
\newcommand{\uCol}{\underline{\mathrm{Col}}}
\newcommand{\Xloc}{\mathcal{X}_{\rm loc}}
\newcommand{\X}{\mathcal{X}}
\newcommand{\be}{\mathbf{z}}
\renewcommand{\div}{\mathrm{div}}
\newcommand{\corank}{\mathrm{corank}}

\newcommand{\Qpnn}{\Qp(\mu_{p^{n+1}})}
\newcommand{\Tam}{\mathrm{Tam}}

  \DeclareFontFamily{U}{wncy}{}
  \DeclareFontShape{U}{wncy}{m}{n}{<->wncyr10}{}
  \DeclareSymbolFont{mcy}{U}{wncy}{m}{n}
  \DeclareMathSymbol{\sha}{\mathord}{mcy}{"58}

\begin{document}

\title[Asymptotic growth of sha for modular forms]{On the asymptotic growth of Bloch--Kato--Shafarevich--Tate groups of modular forms over cyclotomic extensions}

\author[A.~Lei]{Antonio Lei}
\address[Lei]{D\'epartement de math\'ematiques et de statistique\\
Pavillon Alexandre-Vachon\\
Universit\'e Laval\\
Qu\'ebec, QC, Canada G1V 0A6}
\email{antonio.lei@mat.ulaval.ca}

\author[D.~Loeffler]{David Loeffler}
\address[Loeffler]{Mathematics Institute\\
Zeeman Building, University of Warwick\\
Coventry CV4 7AL, UK}
\email{d.a.loeffler@warwick.ac.uk}

\author[S.L.~Zerbes]{Sarah Livia Zerbes}
\address[Zerbes]{Department of Mathematics \\
University College London\\
Gower Street, London WC1E 6BT, UK}
\email{s.zerbes@ucl.ac.uk}
\thanks{The authors are grateful to acknowledge support from the following grants: NSERC Discovery Grants Program 05710 (Lei); Royal Society University Research Fellowship (Loeffler); Leverhulme Research Fellowship (Zerbes).}

\begin{abstract}
We study the asymptotic behaviour of the Bloch--Kato--Shafarevich--Tate group of a modular form $f$ over the cyclotomic $\Zp$-extension of $\QQ$ under the assumption that $f$ is non-ordinary at $p$. In particular, we give upper bounds of these groups in terms of Iwasawa invariants of Selmer groups defined using $p$-adic Hodge Theory. These bounds have the same form as the formulae of Kobayashi, Kurihara and Sprung for supersingular elliptic curves.
\end{abstract}


\maketitle

\section{Introduction}

 Let $p$ be an odd prime and $f$ a normalised new cuspidal modular eigenform of weight $k \ge 2$, and $p$ an odd prime which does not divide the level of $f$. For notational simplicity, we assume in this introduction that all the Fourier coefficients of $f$ lie in $\ZZ$. We let $V_f$ be the \emph{cohomological} $p$-adic Galois representation attached to $f$ (so the determinant of $V_f$ is $\chi^{1-k}$ times a finite-order character). Then $V_f$ has Hodge--Tate weights $\{0, 1-k\}$, where our convention\footnote{This is usual in $p$-adic Hodge theory, but the opposite convention appears to be common in papers on modularity lifting.} is that the Hodge--Tate weight of the cyclotomic character is $1$.  Let $T_f$ be the canonical $G_{\QQ}$-stable $\Zp$-lattice in $V_f$ defined by Kato \cite[8.3]{kato04}.
  
 Let $K_\infty$ be the cyclotomic $\Zp$-extension of $\QQ$ and write $K_n$ for the unique  sub-extension of degree $p^n$. Our aim is to study the asymptotic behaviour of the Bloch--Kato--Shafarevich--Tate groups $\sha(K_n, T_f(j))$ (with $j\in [1,k-1]$), whose definition we shall recall below.

 When $k = 2$, the form $f$ corresponds to an isogeny class of elliptic curves, and we may choose a curve $\cE$ in this isogeny class such that $T_f(1) = T_p(\cE)$, where the latter is the $p$-adic Tate module of $\cE$. In this case it can be shown that the group $\sha(K_n, T_f(1))$ is the quotient of the classical $p$-primary Shafarevich--Tate group $\sha_p(K_n, \cE)$ by its maximal divisible subgroup; hence if the latter group is finite (which is a well-known conjecture), the two groups are equal.
  
 \subsection*{The ordinary case} 
  
  The behaviour of the Selmer and Shafarevich--Tate groups over the cyclotomic extension depends sharply on whether $\cE$ has ordinary or supersingular reduction at $p$. If $\cE$ is ordinary, then the $p$-Selmer group 
  \[ \Sel_p(K_\infty, \cE) = \varinjlim_n \Sel_p(K_n, \cE)\]
  of $A$ over $K_\infty$ is cotorsion over the Iwasawa algebra $\Zp\lb\Gal(K_\infty/\QQ)\rb$, by a theorem of Kato \cite[Theorem 17.4]{kato04}. By Mazur's control theorem \cite{mazur72}, this implies that \textbf{if} the groups $\sha_p(K_n, \cE)$ are finite for all $n$, then we must have
  \[ \len_{\Zp}\sha_p(K_n, \cE)= \mu p^n +\lambda n + O(1),\]
  for some Iwasawa invariants $\mu$ and $\lambda$ associated to $\Sel_p(\cE / K_\infty)$.
 
 \subsection*{The supersingular case} 
  
  The case of supersingular elliptic curves with $a_p(\cE) = 0$ has been studied by Kurihara \cite{kurihara02} and Kobayashi \cite{kobayashi03}. Suppose that $\sha_p(K_n, \cE)$ is finite for all $n$ and write $s_n(\cE)=\len_{\Zp}\sha_p(K_n, \cE)$. They showed that for $n$ sufficiently large,
  \[
   s_n(\cE)-s_{n-1}(\cE)=q_n+\lambda_\pm+\mu_\pm (p^n-p^{n-1})-r_\infty(\cE) ,
  \]
  where $q_n$ is an explicit sum of powers of $p$, $r_\infty(\cE)$ is the rank of $\cE$ over $K_\infty$, $\lambda_\pm$ and $\mu_\pm$ are the Iwasawa invariants of some cotorsion signed Selmer groups, and the sign $\pm$ depends on the parity of $n$.
 
  For supersingular elliptic curves with $a_p(\cE) \ne 0$ (which can only occur when $p = 2$ or $3$), Sprung \cite{sprung12} proved a similar formula:
  \[
   s_n(\cE)-s_{n-1}(\cE)=q_n^\star+\lambda_\star+\mu_\star (p^n-p^{n-1})-r_\infty(\cE) ,
  \]
  for $n\gg0$, where $q_n^\star$ is again an explicit sum of powers of $p$, $\star\in\{\#,\flat\}$, $\lambda_\star$ and $\mu_\star$ are Iwasawa invariants of some cotorsion Selmer groups defined in \cite{sprung09} and the choice of $\star$ depends on the ``modesty algorithm''. An analytic version of this formula has been generalised to arbitrary weight 2 modular forms in \cite{sprung15}.

 \subsection*{Higher weights} 
 
  The main result of the present article is that a similar formula for modular forms of higher weight would give us an upper bound on the growth of the Bloch--Kato--Shafarevich--Tate groups. Suppose that $\ord_p(a_p(f))>\frac{k-1}{2p}$ and $3\le k\le p$, where $a_p(f)$ is the $p$-th Fourier coefficient of the modular form $f$. We shall see below that the Selmer coranks
  \[ r_n(f) = \corank_{\Zp} \Sel(K_n, T_f(j)) \]
  stabilise for $n \gg 0$, and we define $r_\infty(f)$ to be the limiting value (see Proposition~\ref{prop:stabilize}). We define
  \[ s_n(f)= \len_{\Zp}\sha_p(K_n, T_f(j))\]
  (which is finite by definition). We prove the inequality (see Theorem~\ref{thm:final} for the precise statement)
  \[
   s_n(f)-s_{n-1}(f)\le q_n^\star+\lambda_\star+\mu_\star (p^n-p^{n-1})+\kappa-r_\infty(f),
  \]
  for $n\gg0$, where $q_n^\star$ is once again a sum of powers of $p$ that depends on $k$ and the parity of $n$, $\lambda_\star$ and $\mu_\star$  are the Iwasawa invariants of the Selmer groups defined in \cite{leiloefflerzerbes10} for some choice of basis of the Wach module of $T_f$, $\kappa$ is some integer that depends on the image of some Coleman maps that we shall review in \S\ref{S:wach} of this article and the choice of $\star$ is given by an explicit algorithm (similar to the ``modesty algorithm'' of Sprung).
  
  The fact that we have an inequality is a result of the growth of the logarithmic matrix contributed from the twists of $T_f(i)$ for $i\ne j$. In the appendix to this paper, we relate the defect of this inequality to the Tamagamwa numbers of $T_f(j)$ using the method developed by Perrin-Riou in \cite{perrinriou03}.
  
  \subsection*{Acknowledgement}
  
   The authors are grateful to the anonymous referee for many useful comments and suggestions, which improved the paper substantially.

\section{Background from $p$-adic Hodge theory}
 
 We recall the necessary notation and definitions from $p$-adic analysis and $p$-adic Hodge theory. For more details see \cite[\S 1.3]{leiloefflerzerbes11}. We fix (for the duration of this article) a finite extension $E / \Qp$ with ring of integers $\cO$, which will be the coefficient field for all the representations we shall consider.
 
 \subsection{Iwasawa algebras and distribution algebras}
 
  Let $\Gamma=\Gal(\QQ(\mu_{p^\infty})/\QQ)$. This group is isomorphic to a direct product $\Delta\times\Gamma_1$, where $\Delta$ is a finite group of order $p-1$ and $\Gamma_1=\Gal(\QQ(\mu_{p^\infty}) / \QQ(\mu_p))$. We choose a topological generator $\gamma$ of $\Gamma_1$, which determines an isomorphism $\Gamma_1 \cong \Zp$. We also fix a finite extension $E$ of $\Qp$ with ring of integers $\cO$ which will be our field of coefficients (i.e.~we will consider representations of Galois groups on $E$-vector spaces).
 
  We write $\Lambda=\cO\lb \Gamma\rb$, the Iwasawa algebra of $\Gamma$. The subalgebra $\cO\lb\Gamma_1\rb$ can be identified with the formal power series ring $\cO\lb X \rb$, via the isomorphism sending $\gamma_1$ to $1 + X$; this extends to an isomorphism
  \begin{equation}\label{eq:lambda}
   \Lambda=\cO[\Delta]\lb  X\rb.
  \end{equation}
  For a character $\eta$ of $\Delta$ and a $\Lambda$-module $M$, $M^\eta$ denotes its $\eta$-isotypic component, which is regarded as an $\cO\lb X\rb$-module. For $n\ge1$, we write $\Gamma_n$ for the subgroup $\Gal(\QQ(\mu_{p^\infty})/\QQ(\mu_{p^n}))$ and $\Lambda_n=\cO[\Gamma / \Gamma_n]$. Note that
  \[
   \Lambda_n=\cO[\Delta]\lb X\rb/(\omega_{n-1}(X)),
  \]
  where $\omega_{n-1}(X)$ denotes the polynomial $(1+X)^{p^{n-1}}-1$.
 
  We may consider $\Lambda$ as a subring of the ring $\cH$ of locally analytic $E$-valued distributions on $\Gamma$. The isomorphism \eqref{eq:lambda} extends to an identification between $\cH$ and the subring of power series $F \in E[\Delta]\lb X \rb$ which converge on the open unit disc $|X| < 1$.
 
 \subsection{Power series rings}
 
  Let $\AQp = \cO\lb \pi\rb$, where $\pi$ is a formal variable. We equip this ring with a $\cO$-linear \emph{Frobenius endomorphism} $\varphi$, defined by $\pi\mapsto (1+\pi)^p-1$, and with an $\cO$-linear action of $\Gamma$ defined by $\pi\mapsto(1+\pi)^{\chi(\sigma)}-1$ for $\sigma\in\Gamma$, where $\chi$ denotes the $p$-adic cyclotomic character.
  
  The Frobenius $\varphi$ has a left inverse $\psi$, satisfying
  \[ (\varphi \circ \psi)(f)(\pi) = \tfrac{1}{p} \sum_{\zeta: \zeta^p = 1} f\left( \zeta(1 + \pi) - 1 \right).
  \]
  The map $\psi$ is not a morphism of rings, but it is $\cO$-linear, and commutes with the action of $\Gamma$.
  
  We write $\BB^+_{\Qp} = \AQp[1/p] \subset E\lb \pi \rb$, and 
  \[ \Brig=\left\{F(\pi)\in E\lb \pi\rb: F \text{ converges on the open unit disc}\right\},\]
  so there are natural inclusions
  \[ \AQp \hookrightarrow \BB^+_{\Qp} \hookrightarrow \Brig.\]  
  The actions of $\vp$, $\psi$, and $\Gamma$ extend to these larger rings (via the same formulae as before). We shall write $q=\vp(\pi)/\pi\in\AQp$, and $t=\log(1+\pi)\in\Brig$.
 
 \subsection{The Mellin transform}
  \label{sect:mellin}
  
  The action of $\Gamma$ on $1 + \pi \in (\AQp)^{\psi = 0}$ extends to an isomorphism of $\Lambda$-modules
  \begin{align*}
   \fM:\Lambda {\stackrel{\cong}{\longrightarrow}}& (\AQp)^{\psi=0}\\
   1\longmapsto& 1+\pi,
  \end{align*}
  called the \emph{Mellin transform}. This can be further extended to an isomorphism of $\cH$-modules
  \[ \cH {\stackrel{\cong}{\longrightarrow}} (\Brig)^{\psi = 0}\]
  which we denote by the same symbol.
 
  \begin{theorem}\label{thm:mellin}
   For every $n \ge 1$, the Mellin transform induces an isomorphism of $\Lambda$-modules
   \[
    \Lambda_n\cong (\AQp)^{\psi=0} /\vp^n(\pi) (\AQp)^{\psi=0}.
   \] 
  \end{theorem}
 
  \begin{proof}
   If $\mu \in \omega_{n-1}(X) \Lambda$, then $\fM(\mu) \in \varphi^n(\pi) (\Brig)^{\psi = 0}$, by \cite[Theorem 5.4]{leiloefflerzerbes10}. However, $\varphi^n(\pi)$ is a monic polynomial in $\pi$, so if an element of $\AQp$ is divisible by $\varphi^n(\pi)$ in $\Brig$, it is divisible by the same element in $\AQp$. Hence the Mellin transform induces a map $\Lambda_n \to (\AQp)^{\psi=0} /\vp^n(\pi) (\AQp)^{\psi=0}$; and this map is surjective, because the Mellin transform itself is surjective. Since both sides are free $\cO$-modules of the same rank, namely $(p-1)p^n$, it follows that the map must in fact be an isomorphism.
  \end{proof}
 
   We write $\partial$ for the differential operator $(1+\pi)\frac{\mathrm{d}}{\mathrm{d}\pi}$ on $\Brig$, and $\Tw$ for the ring automorphism of $\cH$ defined by $\sigma \mapsto \chi(\sigma)\sigma$ for $\sigma\in \Gamma$. Then one has the compatibility relation
   \[ \fM\circ\Tw = \partial\circ\fM. \]
   
  Let $u=\chi(\gamma)$ be the image of our topological generator $\gamma$ under the cyclotomic character, so that $\Tw$ maps $X$ to $u(1 + X) - 1$. If $m\ge 0$ is an integer, we define $\omega_{n,m}(X)=\omega_n(u^{-m}(1+X)-1)$ and $\tilde{\omega}_{n,m}=\prod_{i=0}^{m}\omega_{n,i}$.
  By exactly the same argument as Theorem~\ref{thm:mellin}, this gives the following isomorphism of $\Lambda$-modules
  \begin{equation}
   \label{eq:mellin}
   \Lambda_{n,m}\coloneqq\Lambda/\tilde{\omega}_{n-1,m}\Lambda\cong 
   (\AQp)^{\psi=0}/\vp^n(\pi^{m+1})(\AQp)^{\psi=0}.
  \end{equation}
  
  We will need below the following technical result, regarding the interaction between Mellin transforms and the Iwasawa invariants of power series. We recall the \emph{Weierstrass preparation theorem}, which states that any $F \in \cO\lb X \rb$ can be factorized uniquely as
  \[ F(X) = \varpi^{\mu(F)} \cdot (X^{\lambda(F)} + \varpi G(X)) \cdot u(X),\]
  where $\varpi$ is a uniformizer of $\cO$, $\lambda(F)$ and $\mu(F)$ are non-negative integers, $G \in \cO[X]$ is a polynomial of degree $< \lambda(F)$, and $u \in \cO\lb X \rb^\times$. The quantities $\lambda(F)$ and $\mu(F)$ are called the \emph{Iwasawa invariants} of $F$.
  
  It is clear that, for $x \in \cO_{\mathbf{C}_p}$ with $\ord_p(x) > 0$, we have the lower bound
  \begin{equation}
   \label{eq:newtonpoly}
   \ord_p F(x) \ge \min\left(\tfrac{\mu + 1}{e}, \tfrac{\mu}{e} + \lambda \ord_p(x)\right),
  \end{equation}
  where $e = 1/\ord_p(\varpi)$ is the absolute ramification degree of $F$. Moreover, if $\ord_p(x)$ is sufficiently small (depending on $F$), this lower bound is an equality (it suffices to take $\ord_p(x) < 1/(e\lambda)$).
  
  \begin{proposition}
   \label{prop:newtonpoly}
   Let $f \in \AQp$, and let $g$ be the unique element of $\Lambda(\Gamma_1)$ such that $\fM(g) = (1 + \pi) \vp(f)$. Then the $\lambda$- and $\mu$-invariants of $f$ (as an element of $\cO\lb \pi \rb$) coincide with those of $g$ (as an element of $\cO\lb X \rb$).
  \end{proposition}
  
  \begin{proof}
   This is a consequence of Proposition 7.2 of \cite{loefflerzerbes10}, which shows that for any $f \in \Brig$ and $g \in \cH$ such that $\fM(g) = (1 + \pi) \vp(f)$, and any real $s$ with $0 < s < 1$, we have $v_s(f) = v_s(g)$, where 
   \[ v_s(f) \coloneqq \inf \{ \ord_p f(x) : \ord_p(x) \ge s\}.\]
   When $f \in \cO\lb X \rb$ and $s$ is sufficiently small, $v_s(f)$ is determined by the Iwasawa invariants of $f$: from the inequality \eqref{eq:newtonpoly} and the discussion following, we have $v_s(f) = \tfrac{1}{e} \mu(f) + \lambda(f) s$ for any $s < \tfrac{1}{e \lambda(f)}$. So the cited proposition implies the equalities $\lambda(f) = \lambda(g)$ and $\mu(f) = \mu(g)$.
  \end{proof}

 \subsection{Crystalline representations and Wach modules}
 
  Fontaine has defined a certain topological $\Qp$-algebra $\BB_{\mathrm{cris}}$, equipped with an action of $G_{\Qp}$, a filtration $\Fil^\bullet$, and a Frobenius endomorphism $\varphi$. 
  
  For any $p$-adic representation $V$ of $G_{\Qp}$, we define the \emph{crystalline Dieudonn\'e module} of $V$ by
  \[ \Dcris(V)  = \left( V \otimes_{\Qp} \BB_{\mathrm{cris}} \right)^{G_{\Qp}}. \]
  The space $\Dcris(V)$ inherits a filtration and a Frobenius endomorphism from those of $\BB_{\mathrm{cris}}$. It is known that $\dim_{\Qp} \Dcris(V) \le \dim_{\Qp} V$, and we say $V$ is \emph{crystalline} if equality holds. If in fact $V$ is an $E$-linear representation, then $\Dcris(V)$ is naturally an $E$-vector space (and its filtration and Frobenius are $E$-linear).

  \begin{definition}
   Let $a \le b$ be integers. A \emph{Wach module} over $\BB^+_{\Qp}$ with weights in $[a, b]$ is a finite free $\BB^+_{\Qp}$-module $N$, equipped with an action of $\Gamma$ and a Frobenius
   \[ \varphi: N[1/\pi] \to N[1/\varphi(\pi)] \]
   compatible with those of $\BB^+_{\Qp}$, satisfying the following conditions:
   \begin{itemize}
    \item $\Gamma$ acts trivially on $N / \pi N$,
    \item $\varphi(\pi^b N) \subseteq \pi^b N$,
    \item if $\varphi^*(\pi^b N)$ is the $\BB^+_{\Qp}$-submodule of $\pi^b N$ generated by $\varphi(\pi^b N)$, then the quotient $\pi^b N / \varphi^*(\pi^b N)$ is killed by $q^{b-a}$.
   \end{itemize}
  \end{definition}
  
  Cf.~\cite[Definition III.4.1]{berger04}. In \emph{op.cit.}~it is shown how to attach to every crystalline $E$-linear representation $V$ of $G_{\Qp}$ a Wach module $\NN(V)$ over $\BB^+_{\Qp}$, in such a way that there is a canonical isomorphism
  \[ \NN(V) \otimes_{\BB^+_{\Qp}} \Brig[1/t] \cong \Dcris(V) \otimes_{E} \Brig[1/t].\]
  Moreover, the definition of Wach modules also makes sense integrally, i.e.~over $\AA^+_{\Qp}$; and we may associate to each $\cO$-lattice $T$ in $V$ that is stable under $G_{\Qp}$ an integral Wach module $\NN(T) \subset \NN(V)$ (Lemme II.1.3 of \emph{op.cit.}).
  
  \begin{definition}
   We say $V$ satisfies the \emph{Fontaine--Laffaille condition} if it is crystalline and has Hodge--Tate weights in $[a, a + (p-1)]$ for some $a \in \ZZ$.
  \end{definition}
  
  If $V$ satisfies the Fontaine--Laffaille condition, and $V$ is irreducible of dimension $\ge 2$, then one has a particularly convenient parametrisation of $G_{\Qp}$-stable lattices in $V$. We say a $\cO$-lattice $M \subset \Dcris(V)$ is a \emph{strongly divisible lattice} if the equality
  \[ \varphi\left( M \cap \Fil^i \Dcris(V) \right) \subset p^i M \]
  holds for all $i \in \ZZ$. Then there is a bijection $T \mapsto \Dcris(T)$ between $G_{\Qp}$-stable lattices in $V$, and strongly divisible lattices in $\Dcris(V)$, given by defining $\Dcris(T)$ to be the image of $\NN(T)$ in $\NN(V) / \pi \NN(V) \cong \Dcris(V)$; cf.~\cite[Propositions V.2.1 \& V.2.3]{berger04}.
  
  We shall need below the following technical result. 
  
  \begin{theorem}\label{thm:conditionsWachbasis}
   Let $T$ be a $G_{\Qp}$-stable $\cO$-lattice in a crystalline $E$-linear representation $V$. Then $\left(\varphi^* \NN(T)\right)^{\psi = 0}$ is a free $\Lambda$-module of rank $d = \dim_E V$. Moreover, if $\{ n_1, \dots, n_d\} $ is an $\AA^+_{\Qp}$-basis of $\NN(T)$ which satisfies the condition
   \[ (\gamma - 1) n_i \in \pi^2 \NN(T) \]
   for all $i$, then $\{ (1 + \pi) \vp(n_i) : i = 1, \dots, d\}$ is a $\Lambda$-module basis of $\left(\varphi^* \NN(T)\right)^{\psi = 0}$.
  \end{theorem}
  
  \begin{proof}
   This is shown in the course of the proof of Theorem 3.5 of \cite{leiloefflerzerbes10}. The condition on the basis modulo $\pi^2$ is the conclusion of Lemma~3.9 in \textit{op. cit.}
  \end{proof}
  
 \subsection{Iwasawa cohomology and the Fontaine isomorphism}
 
  If $V$ is an $E$-linear $p$-adic representation of $G_{\Qp}$, and $T \subset V$ is a $G_{\Qp}$-stable $\cO_E$-lattice, then we define \emph{Iwasawa cohomology} groups by
  \[ H^i_{\Iw}(\Qp(\mu_{p^\infty}), T) = \varprojlim_n H^1(\Qp(\mu_{p^n}), T) \]
  (where the inverse limit is with respect to the corestriction maps). These groups are finitely-generated $\Lambda$-modules, zero unless $i \in \{1, 2\}$. If $H^0(\Qp(\mu_{p^\infty}), T/pT) = 0$, which is the case in our applications below, then $H^2_{\Iw}$ is zero, and $H^1_{\Iw}$ is a free $\Lambda$-module of rank equal to the $\cO$-rank of $T$.
  
  The following theorem is the starting-point for our study of Iwasawa cohomology:
  
  \begin{theorem}[Fontaine--Berger]\label{thm:fontaineiso}
   If $V$ is crystalline with all Hodge--Tate weights $\ge 0$, and $V$ has no non-zero quotient on which $G_{\Qp}$ acts trivially, then there is a canonical $\Lambda$-module isomorphism
   \[ h^1_{T}: \NN(T)^{\psi = 1} \to H^1_{\Iw}(\Qp(\mu_{p^\infty}), T).\]
  \end{theorem}
  
  See \cite[\S II.1]{cherbonniercolmez99}, where it is shown that (for any $T$) there is an isomorphism $\DD(T)^{\psi = 1} \to H^1_{\Iw}(\Qp(\mu_{p^\infty}), T)$ where $\DD(T)$ is the $(\varphi, \Gamma)$-module of $T$; and \cite[\S A]{berger03}, where it is shown that $\NN(T)^{\psi = 1} = \DD(T)^{\psi = 1}$ under the above hypotheses.

\section{Wach modules and Coleman maps}\label{S:wach}

\subsection{Review on the definition of Coleman maps}

  Let $f=\sum a_nq^n$ be a normalised new cuspidal modular eigenform  of weight $k\ge 3$ (note that the case $k=2$ can be dealt with using the method of Sprung in \cite{sprung12}), nebentypus $\varepsilon$ and level $N$ with $(p,N)=1$. We take $E$ to be the completion of the smallest number field containing all the coefficients of $f$ at some fixed prime above $p$. We assume that $f$ is non-ordinary at $p$, and that $k \le p$. We write $T_f$ for the $\cO$-linear representation of $G_\QQ$ associated to $f$ as defined by Kato \cite[8.3]{kato04}. It is crystalline, with Hodge--Tate weights $0$ and $1-k$. We fix an integer $j\in[1,k-1]$ and write $T=T_f(j)$ and $\T=T_f(k-1)$. Note that $T=\T(j-k+1)$. 

  The representation $T / \varpi T$ (where $\varpi$ is a uniformiser of $\cO$) is irreducible as a representation of $G_{\Qp}$, so in particular we have 
  \[ H^0(\Qp(\mu_{p^\infty}), T / \varpi T) = 0.\]

  Both $T_f$ and $\T$ are $G_{\Qp}$-stable $\cO_E$-lattices in crystalline representations of $G_{\Qp}$, so we may consider their Wach modules and Dieudonn\'e modules. By \cite[Proposition III.2.1]{berger04}, there are inclusions of $\Brig$-modules
  \begin{align*}
   \Brig\otimes_{\AQp}\NN(\T)\subset&\Brig\otimes_{\cO}\Dcris(\T),\\
   \Brig\otimes_{\cO}\Dcris(T_f)\subset&\Brig\otimes_{\AQp}\NN(T_f),
  \end{align*}
  where the elementary divisors of the inclusions are given by $1$ and $(t/\pi)^{k-1}$ in both cases.

  \begin{lemma}\label{lem:FLbasis}
   There exists an $\cO$-basis $\fv_1,\fv_2$ of $\Dcris(\T)$ such that $\fv_1\in\Fil^0\Dcris(\T)$ and $\fv_2=\vp(\fv_1)$, where $\vp$ is the Frobenius action on $\Dcris(\T)$. 
  \end{lemma}

\begin{proof}
The Fontaine--Laffaille condition of \cite{fontainelaffaille82} implies that for all integers $i$
\begin{itemize}
\item[(a)] $\Fil^i\Dcris(\T)$ is a direct summand of $\Dcris(\T)$;
\item[(b)] $\vp(\Fil^i\Dcris(\T))\subset p^{i}\Dcris(\T)$;
\item[(c)] $\Dcris(\T)=\sum_ip^{-i}\vp(\Fil^i\Dcris(\T))$.
\end{itemize}
The Hodge--Tate weights of $\T$ are $0$ and $k-1$, so $\Fil^0\Dcris(\T)$ is of rank $1$, say $\Fil^0\Dcris(\T)=\cO\cdot \fv_1$ and (b) says that $\fv_2\coloneqq\vp(\fv_1)\in\Dcris(\T)$. Furthermore, (a) tells us that there exists some  $\fv'\in\Dcris(\T)$ such that
\[
\Dcris(\T)=\cO\cdot\fv_1\oplus\cO\cdot\fv'.
\]
By (c), we have
\[
\Dcris(\T)=\cO\cdot\vp(\fv_1)+p^{k-1}\vp(\Dcris(\T)).
\]
Combing the last two equations gives
\begin{equation}\label{eq:basis}
\Dcris(\T)=\cO\cdot\vp(\fv_1)\oplus \cO\cdot p^{k-1}\vp(\fv').
\end{equation}

Let $D$ be the $\cO$-lattice generated by $­\fv_1$ and $\fv_2$.  Note that \eqref{eq:basis} implies that 
\begin{equation}\label{eq:inclusion}
\fv'\in D+\cO\cdot p^{k-1}\vp(\fv').
\end{equation}
As $\fv_2=\vp(\fv_1)$ and $$\vp^2-\frac{a_p}{p^{k-1}}\vp+\frac{\varepsilon(p)}{p^{k-1}}=0$$ on $\Dcris(\T)$, we have $p^{k-1}\vp(\fv_2)=a_p\fv_2-\varepsilon(p)\fv_1$. In particular, this implies that
$p^{k-1}\vp(D)\subset D$. Hence, we may iterate the inclusion \eqref{eq:inclusion} to deduce that
\[
\fv'\in D+\cO\cdot (p^{k-1}\vp)^n(\fv')
\]
for all $n\ge0$. However, as $f$ is non-ordinary at $p$, $p^{k-1}\vp$ is an $\cO$-operator on $\Dcris(\T)$ with strictly positive slope. This implies that $(p^{k-1}\vp)^n\rightarrow0$ as $n\rightarrow\infty$, which forces that $\fv'\in D$. Hence, $D=\Dcris(\T)$ as required. 
\end{proof}

We fix an $\cO$-basis $\fv_1,\fv_2$ of $\Dcris(\T)$, as given by Lemma~\ref{lem:FLbasis}. Since $\Dcris(\T) = \NN(\T) / \pi \NN(\T)$, this basis can be lifted to a basis $\n_1, \n_2$ of $\NN(\T)$ as an $\AA^+_{\Qp}$-module. There is a change of basis matrix $M\in M_{2\times 2}(\Brig)$ such that 
\begin{equation}\label{eq:changeofbasis}
\begin{pmatrix}\n_1&\n_2\end{pmatrix}
=\begin{pmatrix}\fv_1&\fv_2\end{pmatrix}M
\end{equation}
and $M\equiv I_2\mod \pi$, where $I_2$ is the $2\times 2$ identity matrix.
 We write $v_i=\fv_i\cdot t^{k-j-1}e_{-k+j+1}$, $n_i=\n_i\cdot \pi^{k-j-1}e_{-k+j+1}$, $v_{f,i}=\fv_i\cdot t^{k-1}e_{1-k}$ and $n_{f,i}=\n_i\cdot\pi^{k-1}e_{1-k}$ for the corresponding bases of $\Dcris(T)$, $\NN(T)$, $\Dcris(T_f)$ and $\NN(T_f)$ respectively. Here $e_r$ denotes a basis of the Tate motive $\cO(\chi^r)$ for $r\in\ZZ$. By \cite[proof of Proposition~V.2.3]{berger04} and \cite[Proposition~4.2]{lei15}, we may choose our bases so that 
 \begin{equation}\label{eq:Mmodulo}
 M\equiv I_2\mod \pi^{k-1}
 \end{equation}
 and that the matrices of $\vp$ with respect to $v_{1,f},v_{2,f}$ and $n_{1,f}, n_{2,f}$ are given by
\[
\begin{pmatrix}
0 & -\varepsilon(p)\\
p^{k-1}&a_p
\end{pmatrix} \quad\text{and}\quad
\begin{pmatrix}
0 & -\varepsilon(p)\\
(\delta q)^{k-1}&a_p
\end{pmatrix}
\]
respectively, where $\delta=p/(q-\pi^{p-1})\in(\AQp)^\times$. Then, the matrices of $\vp$ with respect to $\fv_1,\fv_2$ and $\n_1,\n_2$ are given by
\[
A=\begin{pmatrix}
0 & -\frac{\varepsilon(p)}{p^{k-1}}\\
1&\frac{a_p}{p^{k-1}}
\end{pmatrix} \quad\text{and}\quad
P=\begin{pmatrix}
0 & -\frac{\varepsilon(p)}{q^{k-1}}\\
\delta^{k-1}&\frac{a_p}{q^{k-1}}
\end{pmatrix}.
\]

\begin{definition}
We define the \emph{logarithmic matrix} $\Mlog$ (with respect to the chosen bases) to be $\fM^{-1}\left((1+\pi)A\vp(M)\right)$.
\end{definition}

\begin{theorem}\label{thm:freebasis}
Let $\n_1,\n_2$ be the basis of $\NN(\T)$ chosen above. Then, $(1+\pi)\vp(\n_1),(1+\pi)\vp(\n_2)$ form a $\Lambda$-basis of $(\vp^*\NN(\T))^{\psi=0}$. 
\end{theorem}
\begin{proof}
Let $\gamma\in\Gamma$ be a topological generator. Then, \eqref{eq:changeofbasis} tells us that
\[
\begin{pmatrix}
\gamma\cdot\n_1&\gamma\cdot\n_2
\end{pmatrix}
=
\begin{pmatrix}
\fv_1&\fv_2
\end{pmatrix}\gamma(M).
\]
This gives the equation
\[
\begin{pmatrix}
\gamma\cdot\n_1&\gamma\cdot\n_2
\end{pmatrix}
=\begin{pmatrix}
\n_1&\n_2
\end{pmatrix}M^{-1}\cdot\gamma(M).
\]
Hence,  for both $i=1,2$, we have
\[
(1-\gamma)\n_i\in\pi^{k-1}\NN(\T)
\]
thanks to \eqref{eq:Mmodulo}. As we assume that $k\ge 3$, we have in particular
\[
(1-\gamma)\n_i\in\pi^{2}\NN(\T),
\]
which is the condition required in Theorem~\ref{thm:conditionsWachbasis}\footnote{This is the only place where we use the assumption that $k\ge3$.}. Therefore, our result follows.
\end{proof}

Recall from \cite[Remark~3.4]{leiloefflerzerbes10} that for all $z\in\NN(\T)^{\psi=1}$, we have $(1-\vp)z\in (\vp^*\NN(\T))^{\psi=0}$. The latter is free of rank $2$ over $\Lambda$, with basis $(1+\pi)\vp(\n_1),(1+\pi)\vp(\n_2)$ as given by Theorem~\ref{thm:freebasis}. This allows us to define the Coleman maps (again, with respect to our chosen bases) as follows.
\begin{definition}\label{defn:Coleman}
For $i\in\{1,2\}$, we define the $\Lambda$-homomorphisms
$\Col_i:\NN(\T)^{\psi=1}\rightarrow\Lambda$ given by the relation
\[
(1-\vp)z=\sum_{i=1}^2\Col_i(z)\cdot(1+\pi)\vp(\n_i)=\begin{pmatrix}
\fv_1& \fv_2
\end{pmatrix}\cdot\Mlog\cdot\begin{pmatrix}
\Col_1(z)\\ \Col_2(z)\end{pmatrix}.
\]
\end{definition}

Let $h^1_\T:\NN(\T)^{\psi=1}\rightarrow\HIw(\Qp(\mu_{p^\infty}),\T)$ be the $\Lambda$-isomorphism given by Theorem~\ref{thm:fontaineiso}.  By an abuse of notation, we shall write $\Col_1,\Col_2$ for the compositions $\Col_1\circ (h^1_\T)^{-1}$ and $\Col_2\circ (h^1_\T)^{-1}$ as well.

 \subsection{A finite projection of the Coleman maps}

  \begin{definition}
   For each $n\ge1$, we define $H_n=\vp^{n-1}(P^{-1})\cdots\vp(P^{-1})$ and $\sH_n=\fM^{-1}\left((1+\pi)H_n\right)$.
  \end{definition}

\begin{remark}
Note that $H_n \in \AQp$, and $\sH_n \in \Lambda$; and $H_1 = \sH_1 = 1$.
\end{remark}

\begin{lemma}We have the congruence\label{lem:Mlogcongruence}
\[
\Mlog\equiv A^{n}\cdot\sH_n \mod \tilde{\omega}_{n-1,k-2}(X)\cH.
\]
\end{lemma}
\begin{proof}
From \eqref{eq:changeofbasis}, we have the relation
\[
MP=A\vp(M),
\]
which we may rewrite as $M=A\vp(M)P^{-1}$.
On iteration, we have 
\[
M=A^{n-1}\vp^{n-1}(M)\vp^{n-2}(P^{-1})\cdots\vp(P^{-1}) P^{-1}.
\]
By  \eqref{eq:Mmodulo}, we have $\vp^{n-1}(M) = 1 \bmod {\vp^{n-1}(\pi^{k-1})}$, so this implies that
\[
 M\equiv A^{n-1}\vp^{n-2}(P^{-1})\cdots\vp(P^{-1}) P^{-1}\mod\vp^{n-1}(\pi^{k-1}).
\]
This implies that
\[
\vp(M)\equiv A^{n-1} \cdot H_n \mod\vp^{n}(\pi^{k-1}).
\]
Hence the result by \eqref{eq:mellin}.
\end{proof}

\begin{lemma}\label{lem:integralfinitephi}
For all $n\ge1$ and $z\in\NN(\T)^{\psi=1}$, $\left(1\otimes\vp^{-n}\right)\circ(1-\vp)z$ is congruent to an element in $\Lambda_{n,k-2}\otimes\Dcris(\T)$ modulo $\tilde{\omega}_{n-1,k-2}(X)\cH\otimes\Dcris(\T)$.
\end{lemma}
\begin{proof}
By Lemma~\ref{lem:Mlogcongruence} and the equation in Definition~\ref{defn:Coleman}, we have the congruence 
\[(1-\vp)z\equiv\begin{pmatrix}
\fv_1& \fv_2
\end{pmatrix}\cdot A^n\cdot\sH_n\cdot\begin{pmatrix}
\Col_1(z)\\ \Col_2(z)
\end{pmatrix}\mod\tilde{\omega}_{n-1,k-2}(X)\cH\otimes\Dcris(\T) .
\]
If we  apply $\left(1\otimes\vp^{-n}\right)$ to both sides, 
we obtain
\[
\left(1\otimes\vp^{-n}\right)\circ(1-\vp)z\equiv\begin{pmatrix}
\fv_1& \fv_2
\end{pmatrix}\cdot\sH_n\cdot\begin{pmatrix}
\Col_1(z)\\ \Col_2(z)
\end{pmatrix}\mod\tilde{\omega}_{n-1,k-2}(X)\cH\otimes\Dcris(\T) .
\]
As $\sH_n$, $\Col_1(z)$ and $\Col_2(z)$ are all defined over $\Lambda$, we see that $\left(1\otimes\vp^{-n}\right)\circ(1-\vp)z$ is indeed congruent to an element in $\Lambda_{n,p-2}\otimes\Dcris(\T)$.
\end{proof}

This allows us to give the following definition.

\begin{definition}For $n\ge1$, define
\begin{align*}
\uCol_{n}:\HIw(\Qp(\mu_{p^\infty}),\T)\rightarrow& \Lambda_{n,k-2}\otimes\Dcris(\T)\\
z\mapsto&\left(1\otimes\vp^{-n}\right)\circ(1-\vp)\circ (h^1_\T)^{-1}(z)\mod \tilde{\omega}_{n-1,k-2}(X).
\end{align*}
\end{definition}
We recall that $h^1_{\T}$ is an isomorphism by Theorem~\ref{thm:fontaineiso}. Therefore, Lemma~\ref{lem:integralfinitephi} tells us that the map $\uCol_n$ is well-defined.

For an integer $m$, we define the twisting map
\[
\Tw_m\coloneqq\Tw^{-m}\otimes t^{-m}e_m:\cH\otimes\Dcris(\T)\rightarrow\cH\otimes\Dcris(\T(m)).
\]
Consider the twisting map $\Tw^{k-j-1}:\sigma\mapsto\chi^{k-j-1}(\sigma)\sigma$ on $\Lambda$.
Since $k-j-1\le k-1$,  $\Tw^{k-j-1}(\tilde{\omega}_{n-1,k-2}(X))$ is divisible by $\omega_{n-1}(X)$. Hence, $\Tw^{k-j-1}$ induces a natural map $\Lambda_{n,k-2}\rightarrow\Lambda_n$.
Therefore, we may define
\begin{align*}
\uCol_{T,n}:\HIw(\Qp(\mu_{p^\infty}),T)\rightarrow& \Lambda_{n}\otimes\Dcris(T)\\
z\mapsto&\Tw_{-k+j+1}\circ\uCol_n(z\cdot e_{k-j-1})\mod {\omega}_{n-1}(X),
\end{align*}
on identifying $\HIw(\Qp(\mu_{p^\infty}),T)\cdot e_{k-j-1}$ with $\HIw(\Qp(\mu_{p^\infty}),\T)$.

\begin{lemma}\label{lem:linear}
The map $\uCol_{T,n}$ defines a $\Lambda_n$-homomorphism from $H^1(\Qpn,T)$ to $\Lambda_n\otimes\Dcris(T)$.
\end{lemma}

\begin{proof}We note that  $\uCol_{T,n}$ is a $\Lambda$-homomorphism since both $\uCol_{n}$ and $x\mapsto\Tw^{k-j-1}\circ(x\cdot e_{k-j-1})$ are $\Lambda$-linear. The fact that $\uCol_{T,n}$ factors through $H^1(\Qpn,T)$ follows from the equation $\HIw(\Qp(\mu_{p^\infty}),T)_{\Gamma_n}= H^1(\Qpn,T)$ (because of the vanishing of $H^2_{\Iw}(\Qp(\mu_{p^\infty}), T)$).
\end{proof}

We have the explicit formula
\begin{multline}\label{eq:explicitprojection}
 \uCol_{T,n}(z)\equiv 
 \begin{pmatrix}
  v_1& v_2
 \end{pmatrix}\cdot
 \Tw^{k-1-j}\left(
  \sH_n\cdot
  \begin{pmatrix}
   \Col_1(z\cdot e_{k-1-j})\\ \Col_2(z\cdot e_{k-1-j})
  \end{pmatrix}
 \right)\\
 \mod\omega_{n-1}(X)\Lambda\otimes\Dcris(T),
\end{multline}
by Lemma~\ref{lem:Mlogcongruence} and the expansion of $1-\vp$ as given in Definition~\ref{defn:Coleman}.

We now modify the definition of $\uCol_{T,n}$ to define a map that lands in $\Lambda_n$.
For any $u\in\Zp^\times$, we define $\uCol_{T,n,u}:\HIw(\Qp(\mu_{p^\infty}),T)\rightarrow\Lambda_{n}$ to be the composition of $\uCol_{T,n}$ and the linear functional on $\Lambda_{n}\otimes\Dcris(T)\rightarrow\Lambda_{n}$ given by $a\cdot v_1+b\cdot v_2\mapsto a+ub$.
More explicitly, \eqref{eq:explicitprojection} tells us that $\uCol_{T,n,u}$ is given by
\begin{equation}\label{eq:formulaColTnu}
\uCol_{T,n,u}(z)\equiv \begin{pmatrix}1& u\end{pmatrix}\cdot\Tw^{k-1-j}\left(\sH_n\cdot\begin{pmatrix}
\Col_1(z\cdot e_{k-1-j})\\ \Col_2(z\cdot e_{k-1-j})
\end{pmatrix}\right)\mod\omega_{n-1}(X) \Lambda.
\end{equation}
Note that Lemma~\ref{lem:linear} tells us that $\uCol_{T,n,u}$ is $\Lambda_n$-linear.

\subsection{Analysis of Bloch--Kato subgroups via Coleman maps}

If $F$ is a finite extension of $\Qp$, we write $H^1_f(F,T)\subset H^1(F,T)$ for the usual Bloch--Kato subgroup from \cite{blochkato90} and $H^1_{/f}(F,T)$ denotes the quotient $H^1(F,T)/H^1_f(F,T)$. The goal of this section is to study $H^1_{/f}(\Qpn,T)$ via the map $\uCol_{T,n,u}$.

Let $\T^*$ be the $\cO$-linear dual of $\T$. 
For each $n\ge1$, we define the pairing
\begin{align*}
\langle\sim,\sim\rangle_n:H^1(\Qpn,\T)\times H^1(\Qpn,\T^*(1))\rightarrow&\Lambda_n\\
(x,y)\mapsto&\sum_{\sigma\in\Gamma/\Gamma_n}[x,y^{\sigma}]_n\sigma,
\end{align*}
where $[\sim,\sim]_n$ is the standard cup-product pairing 
\[H^1(\Qpn,\T)\times H^1(\Qpn,\T^*(1))\rightarrow\cO.\]
On taking inverse limits, this induces a pairing 
\[
\langle\sim,\sim\rangle:\HIw(\Qp(\mu_{p^\infty}),\T)\times\HIw(\Qp(\mu_{p^\infty}),\T^*(1))\rightarrow\Lambda.
\]
It is semi-linear over $\Lambda$ with respect to the involution  on $\Lambda$ (which we denote by $\tilde{\iota}$) in the following sense:
\[
\langle \sigma x,y\rangle=\sigma\langle x,y\rangle,\quad\langle x, \sigma  y\rangle=\sigma^{\tilde{\iota}} \langle x,y\rangle
\]
We may extend the pairing $\langle\sim,\sim\rangle$ by semi-linearity to
\[
 \left(\cH \otimes_{\cO} \HIw(\Qp(\mu_{p^\infty}),\T)\right)
 \times\left(\cH\otimes_{\cO} \HIw(\Qp(\mu_{p^\infty}),\T^*(1))\right)
\rightarrow\cH,
\]
which is again denoted by $\langle\sim,\sim\rangle$ by an abuse of notation.

Recall that in \cite{perrinriou94}, Perrin-Riou defined the big exponential map
\[
\Omega_{\T^*(1),1}:(\Brig)^{\psi=0}\otimes\Dcris(\T^*(1))\rightarrow \cH\otimes\HIw(\Qp(\mu_{p^\infty}),\T^*(1)).
\]
By \cite[proof of Proposition 4.8]{leiloefflerzerbes11}, for all $z\in\HIw(\Qp(\mu_{p^\infty}),\T)$,
 \[
(\fM^{-1}\otimes1)(1-\vp)z=
\sum_{i=1}^2\langle z,\Omega_{\T^*(1),1}((1+\pi)\otimes \fv_i')\rangle \fv_i
\]
where $\fv_1',\fv_2'$ is the dual basis of $\Dcris(\T^*(1))$ to $\fv_1,\fv_2$ with respect to the natural pairing 
\[
[\sim,\sim]:\Dcris(\T)\times\Dcris(\T^*(1))\rightarrow\cO.
\]
 Therefore,
\begin{align*}
\uCol_n(z)&=\sum_{i=1}^2\langle z,\Omega_{\T^*(1),1}((1+\pi)\otimes \fv_i')\rangle \vp^{-n}(\fv_i)\mod\tilde{\omega}_{n-1,k-2}\\
&=\sum_{i=1}^2\langle z,\Omega_{\T^*(1),1}((1+\pi)\otimes (p\vp)^n(\fv_i'))\rangle \fv_i\mod\tilde{\omega}_{n-1,k-2}
\end{align*}
as the dual of $\vp^{-1}$ with respect to $[\sim,\sim]$ is $p\vp$. This description allows us to make the following choice of $u$ to describe the kernel of $\uCol_{T,n,u}$.

\begin{proposition}\label{prop:choiceofu}
There exists $u\in\Zp^\times$ such that $\ker(\uCol_{T,n,u})=H^1_f(\Qpn,T)$.
\end{proposition}
\begin{proof}Write $v'=(\fv_1'+u\fv_2')\cdot t^{-k+j+1}e_{k-j-1}\in\Dcris(T^*(1))$ and let $z\in H^1(\Qpn,T)$.
If $\theta$ is a Dirichlet character of conductor $p^m>1$, we have the interpolation formula of Perrin-Riou \cite[\S 3.2.3]{perrinriou94} (see also \cite[\S3.2]{lei11})
\begin{equation}\label{eq:zeros}
\frac{\theta\left(\uCol_{T,n,u}(z)\right)}{(-1)^{k-j-1}(k-j-1)!}=\sum_{\sigma\in\Gamma/\Gamma_m }\frac{\theta^{-1}(\sigma)}{\tau(\theta)}[\exp^*_{T,m}(z^\sigma),p^n\vp^{n-m}(v')],
\end{equation}
where $\exp^*_{T,m}:H^1(\Qp(\mu_{p^m}),T)\rightarrow \Qp(\mu_{p^m})\otimes\Fil^0\Dcris(T)$ is the Bloch--Kato dual exponential map and $\tau(\theta)$ is the Gauss sum of $\theta$. There is a similar formula when $\theta$ is the trivial character on replacing $\vp^{-m}$ by $\left(1-\frac{\vp^{-1}}{p}\right)(1-\vp)^{-1}$. We note that here $\exp^*_{T,m}(z)$ is the shorthand for $\exp^*_{T,m}\circ\cor_{n/m}(z)$, where $\cor_{n/m}$ denotes the the corestriction map $H^1(\Qpn,T)\rightarrow H^1(\QQ(\mu_{p^m}),T)$. Recall that $\exp^*_{T,n}$ factors through $H^1_{/f}(\Qpn,T)$. Therefore, we see that $H^1_f(\Qpn,T)$ is contained in $\ker(\uCol_{T,n,u})$.

We choose $u$ so that $\vp^{n-m}(v')$, $1\le m\le n$ and $\vp^{n}\left(1-\frac{\vp^{-1}}{p}\right)(1-\vp)^{-1}(v')$  are not contained inside $\Fil^0\Dcris(V)$. We note that such $u$ exists since all maps are surjective on $\Dcris(V)$ and $\Fil^0\Dcris(V)$ is of dimension one. Let $v''$ be any $\cO$-basis of $\Dcris(T)/\Fil^0\Dcris(T)$. In particular, for each $m\ge1$, there exists a non-zero constant $c_m\in\cO$ such that $\vp^{n-m}(v')\equiv c_m v''$ and $\vp^{n}\left(1-\frac{\vp^{-1}}{p}\right)(1-\vp)^{-1}(v')\equiv c_0v''$ modulo $\Fil^0\Dcris(T)$.

 Suppose that $\uCol_{T,n,u}(z)=0$. From \eqref{eq:zeros}, we deduce that
 \[
  \sum_{\sigma\in \Gamma/\Gamma_n}\theta^{-1}(\sigma)[\exp^*_{T,n}(z^\sigma),v'']=0 
 \]
 for all characters $\theta$ on $\Gamma/\Gamma_n$. By the independence of the characters, this implies that $[\exp^*_{T,n}(z^\sigma),v'']=0$ for all $\sigma$. In particular, $z$ is contained in the kernel of $\exp_{T,n}^*$, which is $H^1_f(\Qpn,T)$.
\end{proof}

\begin{corollary}\label{cor:injection}
For any $u\in\Zp^\times$ that satisfies the condtion of Proposition~\ref{prop:choiceofu}, $\uCol_{T,n,u}$ induces an injection of $\Lambda_n$-modules
\[
H^1_{/f}(\Qpn,T)\hookrightarrow \Lambda_n,
\]
whose cokernel is finite.
\end{corollary}
\begin{proof}
The injectivity is given by Proposition~\ref{prop:choiceofu}. By \cite[Theorem~4.1]{blochkato90}, $H^1_f(\Qp(\mu_{p^n}),V)$ is isomorphic to $\Dcris(V)/\Fil^0\Dcris(V)\otimes_{\Zp} \Qp(\mu_{p^n})$. Hence, by duality $H^1_{/f}(\Qp(\mu_{p^n}),V)$ is isomorphic to $\Fil^0\Dcris(V)\otimes_{\Zp} \Qp(\mu_{p^n})$. Therefore, the finiteness of the cokernel follows from the fact that the two sides have the same $\Zp$-rank.
\end{proof}
We remark that our map $\uCol_{T,n,u}$ does depend on the choice of $u$. But it does not affect our calculations later, see the proof of Proposition~\ref{pro:evalXloc} below.
\section{Results on $p$-adic valuations}

\subsection{Review of Kobayashi rank}

Given an $\cO$-module $N$, we shall write $\len(N)$ for the $\cO$-length of $N$.  We fix a family of primitive $p^n$-th root of unity $\zeta_{p^n}$ and write $\epsilon_n=\zeta_{p^n}-1$.

\begin{definition}
Let $N=(N_n)$ be an inverse system of finitely generated $\cO$-modules with transition maps $\pi_n:N_n\rightarrow N_{n-1}$. If $\pi_n$ has finite kernel and cokernel, the Kobayashi rank $\nabla N_n$ is defined as
\[
\nabla N_n\coloneqq\len(\ker\pi_n)-\len(\coker\pi_n)+{\rm rank}_{\cO}N_{n-1}.
\]
If $L$ is an $\cO\lb X\rb$-module, we define $\nabla_n L$ to be  $\nabla\left( L/\omega_n(X) L\right)$, with the connecting map given by the natural projection $L/\omega_n(X) L\rightarrow L/\omega_{n-1}(X) L$, if its kernel and cokernel are finite.
\end{definition}
%

\begin{lemma}\label{lem:kob}
Let $F\in\cO\lb X\rb$ be a non-zero element. Let $N$ be the inverse limit defined by $N_n=\cO\lb X\rb/(F,\omega_n)$, where the the connecting maps are the natural projections.
\begin{itemize}
\item[(a)]  Suppose that $F(\epsilon_n)\ne0$, then $\nabla  N_n$ is defined and is equal to $\ord_{\epsilon_n}F(\epsilon_n)$.
\item[(b)] When $n$ is sufficiently large, then $\nabla N_n$ is defined. Furthermore,
\[
\nabla N_n=e\times\ord_{\epsilon_n}F(\epsilon_n)=e\lambda(F)+(p^n-p^{n-1})\mu(F),
\]
where $e$ is the ramification index of $E/\Qp$ and $\lambda(F)$, $\mu(F)$ are the Iwasawa invariants as defined in \S \ref{sect:mellin} above.
\item[(c)] If $L$ is a finitely generated torsion $\cO\lb X\rb$-module, then $\nabla_nL$ is defined for $n\gg0$ and its value is given by
\[
\lambda(L)+(p^n-p^{n-1})\mu(L),
\]
where $\lambda(L)$ and $\mu(L)$ are the $\lambda$- and $\mu$-invariants of a generator of the characteristic ideal of $L$.
\end{itemize}
\end{lemma}
\begin{proof}
This follows from the same proof as \cite[Lemma~10.5]{kobayashi03}.
\end{proof}

We write $p^r$ for the size of the residue field of $E$. The following lemma allows us to relate the growth in the size of a tower of finite $\cO$-modules and Kobayashi ranks.

\begin{lemma}\label{lem:finitekob}
Suppose that $N=(N_n)$ is an inverse limit of finite $\cO$-modules such that $|N_n|=p^{\e_n}$ for some integer $\e_n\in r\ZZ$ for all  $n\ge1$. Then, $r\nabla N_n=\e_n-\e_{n-1}$.
\end{lemma}
\begin{proof}
Since $N_{n-1}$ is finite, we have
\begin{align*}
\nabla N_n&=\len(\ker\pi_n)-\len(\coker\pi_n)\\
&=(\len(N_n)-\len(\image\pi_n))-(\len (N_{n-1})-\len(\image\pi_n))\\
&=\len(N_n)-\len(N_{n-1}).
\end{align*}
In general, if $L$ is a finite $\cO$-module, then $|L|=p^{r\len(L)}$. Hence the result.
\end{proof}

Finally, we prove a lemma on $p$-adic valuations that will be needed later.

\begin{lemma}\label{lem:twist}
Let $F\in\cO\lb X\rb$ be non-zero. Then for all sufficiently large integers $n$ we have 
\[
 \ord_{p}F(\epsilon_n)=\ord_{p}\fM(F)(\epsilon_{n+1}).
\]
Moreover, for $n \gg 0$ we also have
\[
 \ord_{p}F(\epsilon_n)=\ord_{p}\Tw(F)(\epsilon_n).
\]
\end{lemma}
\begin{proof}
We may write $\fM(F)=(1+\pi)\vp(G)$ for some $G\in \AQp$. By Proposition \ref{prop:newtonpoly}, $F$ and $G$ have the same Iwasawa invariants, so $\ord_p F(\epsilon_n) = \ord_p G(\epsilon_n)$ for $n \gg 0$. This implies the first part of the lemma since $(1+\pi)\vp(G)(\epsilon_{n+1}) = \zeta_{p^{n+1}}G(\epsilon_{n})$. The second part of the lemma follows from the fact that $\Tw$ preserves $\mu$- and $\lambda$-invariants.
\end{proof}
\subsection{Calculations on evaluation matrices}

From now on, we shall write $v=\ord_p(a_p)$, where $a_p$ is the $p$-th Fourier coefficient of $f$. Following \cite[\S4.1]{sprung12}, given any $2\times2$ matrix $\phi=\begin{pmatrix}
a&b\\c&d
\end{pmatrix} $ defined over $\overline{\Qp}$, we write $\ord_p(\phi)=\begin{pmatrix}
\ord_p(a)&\ord_p(b)\\ \ord_p(c)&\ord_p(d)
\end{pmatrix}.$

\begin{lemma}\label{lem:evaluationP}
Let $1\le i\le n-2$, then $$\ord_p\left(\vp^{i}(P^{-1})(\epsilon_n)\right)=\begin{pmatrix}
v&0\\
\frac{k-1}{p^{n-i-1}}&\infty
\end{pmatrix}.$$
\end{lemma}
\begin{proof}
Recall that \[
P=\begin{pmatrix}
0 & -\frac{\varepsilon(p)}{q^{k-1}}\\
\delta^{k-1}&\frac{a_p}{q^{k-1}}
\end{pmatrix},\]
so its inverse is given by
\[
P^{-1}=\begin{pmatrix}
\frac{a_p}{\delta^{k-1}\varepsilon(p)} & \frac{1}{\delta^{k-1}}\\
-\frac{q^{k-1}}{\varepsilon(p)}&0
\end{pmatrix}.
\]
Therefore, our result follows from the fact that $\delta\in\Zp^\times$,  $\varepsilon(p)\in\cO^\times$ and $\vp^i(q)$ is equal to the $p^{i+1}$-cyclotomic polynomial, so $\vp^i(q)(\epsilon_n)=\frac{\zeta_{p^{n-i-1}}-1}{\zeta_{p^{n-i}}-1}$ whose $p$-adic valuation is $1/p^{n-i-1}$.
\end{proof}

\begin{proposition}\label{prop:evaluateH}
Assume that $2v> \frac{k-1}{p}$. For all $n\ge1$,  $$\ord_p\left(H_n(\epsilon_n)\right)=
\begin{cases}
\begin{pmatrix}
v+\sum_{i=1}^{\frac{n-1}{2}}\frac{k-1}{p^{2i-1}}&\sum_{i=1}^{\frac{n-1}{2}}\frac{k-1}{p^{2i}}\\
\infty&\infty
\end{pmatrix}
&\text{if $n$ is odd.}\\
\begin{pmatrix}
\sum_{i=1}^{\frac{n}{2}}\frac{k-1}{p^{2i-1}}&v+\sum_{i=1}^{\frac{n}{2}-1}\frac{k-1}{p^{2i}}\\
\infty&\infty
\end{pmatrix}&
\text{if $n$ is even.}
\end{cases}
$$
\end{proposition}
\begin{proof}
By Lemma~\ref{lem:evaluationP}, we have
\[
\ord_p\left(H_n(\epsilon_n)\right)=
\begin{pmatrix}
v&0\\
\infty&\infty
\end{pmatrix}\begin{pmatrix}
v&0\\
\frac{k-1}{p}&\infty
\end{pmatrix}\cdots\begin{pmatrix}
v&0\\
\frac{k-1}{p^{n-1}}&\infty
\end{pmatrix}.
\]
In particular,
\begin{equation}\label{eq:induction}
\ord_p\left(H_{n+1}(\epsilon_{n+1})\right)=
\ord_p\left(H_n(\epsilon_n)\right)
\begin{pmatrix}
v&0\\
\frac{k-1}{p^n}&\infty
\end{pmatrix}.
\end{equation}
Therefore, 
\[
\ord_p(H_1(\epsilon_1))=\begin{pmatrix}
v&0\\
\infty&\infty
\end{pmatrix}\quad\text{and}\quad\ord_p(H_2(\epsilon_2))=\begin{pmatrix}
\frac{k-1}{p}&v\\
\infty&\infty
\end{pmatrix}
\]
since $2v>\frac{k-1}{p}$ by our assumption.

Suppose that \begin{align*}
\ord_p(H_{2\ell-1}(\epsilon_{2\ell-1}))&=\begin{pmatrix}
v+\sum_{i=1}^{\ell-1}\frac{k-1}{p^{2i}}&\sum_{i=1}^{\ell-1}\frac{k-1}{p^{2i-1}}\\
\infty&\infty
\end{pmatrix},\\
\ord_p(H_{2\ell}(\epsilon_{2\ell}))&=\begin{pmatrix}
\sum_{i=1}^\ell\frac{k-1}{p^{2i-1}}&v+\sum_{i=1}^{\ell-1}\frac{k-1}{p^{2i}}\\
\infty&\infty\end{pmatrix}
\end{align*}
for some integer $\ell\ge1$. By \eqref{eq:induction},
we have first of all
\[ 
\ord_p(H_{2\ell+1}(\epsilon_{2\ell+1}))=\begin{pmatrix}
v+\sum_{i=1}^{\ell}\frac{k-1}{p^{2i}}&\sum_{i=1}^{\ell}\frac{k-1}{p^{2i-1}}\\
\infty&\infty
\end{pmatrix}
\]
because $\sum_{i=1}^{\ell}\frac{k-1}{p^{2i}}<\sum_{i=1}^{\ell}\frac{k-1}{p^{2i-1}}$. On applying \eqref{eq:induction} again, we have
\[
\ord_p(H_{2\ell+2}(\epsilon_{2\ell+2}))=\begin{pmatrix}
\sum_{i=1}^{\ell+1}\frac{k-1}{p^{2i-1}}&v+\sum_{i=1}^{\ell}\frac{k-1}{p^{2i}}\\
\infty&\infty\end{pmatrix}
\]
thanks to our assumption that $2v>\frac{k-1}{p}$, which implies that
\[
2v+\sum_{i=1}^\ell\frac{k-1}{p^{2i}}>\sum_{i=1}^{\ell+1}\frac{k-1}{p^{2i-1}}.
\]
Therefore, our result follows from induction.
\end{proof}

For $i=1,2$, we fix two elements $F_1,F_2\in\cO\lb X\rb$ with  $\mu_i$ and $\lambda_i$ being its $\mu$- and $\lambda$-invariants.

\begin{corollary}Under the condition that $2v>\frac{k-1}{p}$, for $n\gg0$ we have the formulae
\[
\ord_{\epsilon_n}\left((\sH_{n+1})_{1,1}\cdot F_1(\epsilon_n)\right)=
\begin{cases}
\lambda_1+(p^n-p^{n-1})\left(\frac{\mu_1}{e}+v+\sum_{i=1}^{\frac{n-1}{2}}\frac{k-1}{p^{2i-1}}\right)&\text{$n$ odd,}\\
\lambda_1+(p^n-p^{n-1})\left(\frac{\mu_1}{e}+\sum_{i=1}^{\frac{n}{2}}\frac{k-1}{p^{2i-1}}\right)&\text{$n$ even,}
\end{cases}
\]
\[
\ord_{\epsilon_n}\left((\sH_{n+1})_{1,2}\cdot F_2(\epsilon_n)\right)=
\begin{cases}
\lambda_2+(p^n-p^{n-1})\left(\frac{\mu_2}{e}+\sum_{i=1}^{\frac{n-1}{2}}\frac{k-1}{p^{2i}}\right)&\text{$n$ odd,}\\
\lambda_2+(p^n-p^{n-1})\left(\frac{\mu_2}{e}+v+\sum_{i=1}^{\frac{n}{2}-1}\frac{k-1}{p^{2i}}\right)&\text{$n$ even.}
\end{cases}
\]
\end{corollary}
\begin{proof}
By Lemma~\ref{lem:twist}, $\ord_{p}\sH_{n+1}(\epsilon_n)=\ord_pH_n(\epsilon_n)$. Hence, our result follows from combining Proposition~\ref{prop:evaluateH} with Lemma~\ref{lem:kob}(b).
\end{proof}

\begin{corollary}\label{cor:modesty}
Suppose that $2v>\frac{k-1}{p}$. For $n\gg0$ and $n$ odd, we have
\begin{align*}
\ord_{\epsilon_n}\left((\sH_{n+1})_{1,1}\cdot F_1(\epsilon_n)\right)<\ord_{\epsilon_n}\left((\sH_{n+1})_{1,2}\cdot F_2(\epsilon_n)\right)&&\text{if $\frac{\mu_1}{e}+v+\frac{k-1}{p+1}\le\frac{\mu_2}{e}$}\\
\ord_{\epsilon_n}\left((\sH_{n+1})_{1,1}\cdot F_1(\epsilon_n)\right)>\ord_{\epsilon_n}\left((\sH_{n+1})_{1,2}\cdot F_2)(\epsilon_n)\right)&&\text{if $\frac{\mu_1}{e}+v+\frac{k-1}{p+1}>\frac{\mu_2}{e}$}.
\end{align*}
For $n\gg 0$ and $n$ even, we have
\begin{align*}
\ord_{\epsilon_n}\left((\sH_{n+1})_{1,1}\cdot F_1(\epsilon_n)\right)<\ord_{\epsilon_n}\left((\sH_{n+1})_{1,2}\cdot F_2(\epsilon_n)\right)&&\text{if $\frac{\mu_1}{e}<\frac{\mu_2}{e}+v+\frac{k-1}{p+1}$}\\
\ord_{\epsilon_n}\left((\sH_{n+1})_{1,1}\cdot F_1(\epsilon_n)\right)>\ord_{\epsilon_n}\left((\sH_{n+1})_{1,2}\cdot F_2(\epsilon_n)\right)&&\text{if $\frac{\mu_1}{e}\ge\frac{\mu_2}{e}+v+\frac{k-1}{p+1}$}.
\end{align*}
\end{corollary}
\begin{proof}
Note that
\[
\sum_{i=1}^{\frac{n-1}{2}}\frac{k-1}{p^{2i-1}}-\sum_{i=1}^{\frac{n-1}{2}}\frac{k-1}{p^{2i}}>0\quad\text{and}
\quad\sum_{i=1}^{\frac{n}{2}}\frac{k-1}{p^{2i-1}}-\sum_{i=1}^{\frac{n}{2}-1}\frac{k-1}{p^{2i}}>0
\]
and that both sequences are strictly increasing and tend to $\frac{k-1}{p+1}$ as $n\rightarrow\infty$. Hence the result.
\end{proof}

\subsection{Some global Iwasawa modules}

 For $n \ge 0$ let us write $K_n = \QQ(\mu_{p^n})$.

 \begin{definition}[{cf.~\cite[\S 12.2]{kato04}}]
  For $m \ge 0$, we define
  \[ \HH^m(T) \coloneqq \varprojlim_n H^m_{\text{\textup{\'et}}}\Big(\Spec \cO_{K_n}[1/p], j_* T\Big),\]
  where the inverse limit is respect to the corestriction maps, and $j$ is the inclusion map $\Spec K_n \hookrightarrow \Spec \cO_{K_n}[1/p]$.
 \end{definition}
 
 By \cite[12.4(3)]{kato04}, the modules $\HH^m(T)$ are finitely-generated over $\Lambda$, and are zero unless $m \in \{1, 2\}$; and $\HH^1(T)$ is free of rank 1 over $\Lambda$. We fix an element $\be\in\HH^1(T)$ so that $\HH^1(T)=\Lambda\cdot \be$. Tensoring with the basis vector $e_{k - 1 - j}$ of $\cO(k - 1 - j)$ gives a bijection
 \[ \HH^1(T) \cong \HH^1(\T), \]
 and (in a slight abuse of notation) we shall write $\Col_i(\be)$ for the image of $\be\cdot e_{k-1-j}$ under $\Col_i$ composed with the localization map $\HH^1(\T)\rightarrow \HIw(\Qp(\mu_{p^\infty}),\T)$.
 
 \begin{definition}\label{defn:modesty}
 For $i=1,2$ and $\eta$ a Dirichlet character modulo $p$. Let $\mu_i^\eta$ be the $\mu$-invariant of $\Col_i(\be)^\eta$. For each $n\ge1$, we define an integer $\tau(n,\eta)\in\{1,2\}$ by
 \[
 \begin{cases}
 1&\text{if $\frac{\mu_1^\eta}{e}+v+\frac{k-1}{p+1}\le\frac{\mu_2^\eta}{e}$ and $n$ odd or $\frac{\mu_1^\eta}{e}<\frac{\mu_2^\eta}{e}+v+\frac{k-1}{p+1}$ and $n$ even,}\\
 2&\text{otherwise.}
 \end{cases}
 \]
 Furthermore, we write $q_{n}^*=\ord_{\epsilon_n}\left((\sH_{n+1})_{1,\tau(n,\eta)}(\epsilon_n)\right)$.
 \end{definition}
 
 Note in particular that $q_n^*$ is a sum of some powers of $p$, together with possibly $v$, as given by Proposition~\ref{prop:evaluateH}. Furthermore, Corollary~\ref{cor:modesty} tells us that
 \begin{equation}\label{eq:valofz}
 \ord_{\epsilon_n}\left(\sum_{i=1}^2(\sH_{n+1})_{1,i}\cdot\Col_i(\be)^\eta(\epsilon_n)\right)=q_n^*+\ord_{\epsilon_n}\Col_{\tau(n,\eta)}(\be)^{\eta}(\epsilon_n).
 \end{equation}

\subsection{Analysis of some local Iwasawa modules}

For $n\ge1$, we define
\[
\Xloc(\QQ(\mu_{p^n}))=\coker\left(\HH^1(T)_{\Gamma_n}\rightarrow H^1_{/f}(\Qp(\mu_{p^n}),T)\right),
\]
which gives an inverse limit with the connecting maps given by the corestriction maps. We would like to study $\nabla \Xloc(\QQ(\mu_{p^{n+1}}))^\eta$ for a fixed Dirichlet character $\eta$ modulo $p$.

\begin{proposition}\label{pro:evalXloc}
Suppose that $\uCol_1(\be)^\eta$ and $\uCol_2(\be)^\eta$ are non-zero. For $n\gg0$, $\nabla \Xloc(\QQ(\mu_{p^{n+1}}))^\eta$ is defined, and its value is bounded above by
\[\nabla_n\Xloc^\eta\le eq_n^*+\nabla_n(\cO\lb X\rb/\Col_{\tau(n,\eta)}(\be)^\eta).\]
\end{proposition}
\begin{proof}
Recall from Corollary~\ref{cor:injection}, we have the injection
\[
\uCol_{T,n+1,u}:H^1_{/f}(\Qpnn,T)\hookrightarrow \Lambda_{n+1}.
\]
On taking $\Gamma_{n}$-coinvariants, the \textit{same} map (\textit{not} $\uCol_{T,n,u}$) induces an injection
\[
\uCol_{T,n+1,u}:H^1_{/f}(\Qpn,T)\hookrightarrow \Lambda_{n},
\]
which admits the same description as \eqref{eq:formulaColTnu}. We write $\coker_{n+1}$ and $\coker_{n}$ for the cokernels of these two maps respectively. Then, we have the commutative diagram
\[
\begin{CD}
0 @>>> H^1_{/f}(\Qpnn,T) @>\uCol_{T,n+1,u}>>\Lambda_{n+1} @>>>\coker_{n+1}@>>>0\\
@.  @VVV @VVV @V \pi VV\\
0 @>>> H^1_{/f}(\Qpn,T) @>\uCol_{T,n+1,u}>> \Lambda_{n}@>>>\coker_{n}@>>>0,
\end{CD}
\]
where the vertical maps are all natural projections. This gives
\[
\begin{CD}
0 @>>> \Xloc(\Qpnn) @>>>\Lambda_{n+1}/(\uCol_{T,n+1,u}(\be)) @>>>\coker_{n+1}@>>>0\\
@.  @VVV @VVV @VV \pi V\\
0 @>>> \Xloc(\Qpn) @>>> \Lambda_{n}/(\uCol_{T,n+1,u}(\be))@>>>\coker_{n}@>>>0.
\end{CD}
\]
Recall from Corollary~\ref{cor:injection} that $\coker_{n+1}$ is finite (in particular, $\coker_{n}$ too). Hence, on taking $\eta$-isotypic components, $\nabla\coker_{n+1}^\eta$ (with respect to $\pi$) is defined. In fact, it is given by $\len(\ker\pi^\eta)$, which is $\ge0$.

Furthermore, recall that we assume $\Col_i(\be)^\eta\ne0$ for $i=1,2$. Proposition~\ref{prop:evaluateH} tells us that the second row of $\sH_{n+1}(\epsilon_n)$ is $0$. So, the formulae \eqref{eq:formulaColTnu} and \eqref{eq:valofz} imply that $\uCol_{T,n+1,u}(\be)(\epsilon_n)\ne0$ when $n\gg0$. Hence, $\nabla\left(\Lambda_{n+1}/(\uCol_{T,n+1,u}(\be))\right)^\eta=\nabla_n \left(\cO\lb X \rb/\uCol_{T,n+1,u}(\be)^\eta\right)$ is defined. Its value is given by
\[
eq_n^*+\nabla_n(\cO\lb X\rb/\Col_{\tau(n,\eta)}(\be)^\eta),
\]
thanks to Lemma~\ref{lem:kob}.

Therefore, the fact that the Kobayashi rank $\nabla$ respects short exact sequences (\cite[Lemma~10.4]{kobayashi03}) tells us that $\nabla \Xloc(\Qpnn)^\eta$ is defined and its value is equal to
\[
\nabla_n \left(\cO\lb X \rb/\uCol_{T,n+1,u}(\be)^\eta\right)-\len(\ker\pi^\eta).
\]
Hence the result.
\end{proof}

This can be considered as a weakened version of the modesty proposition  \cite[Proposition~3.10]{sprung12}. In the $k=2$ case, equality holds because the projection $\pi$ turns out to be an injection (see \cite[Lemma~10.7]{kobayashi03} and \cite[Lemma~4.12]{sprung12}).

\section{Selmer groups and Shafarevich--Tate groups}

\subsection{Signed Selmer groups}
Let $T^\vee$ be the Pontryagin dual of $T$. As in \cite{leiloefflerzerbes10}, the Coleman maps allow us to define the Selmer groups
\[
\Sel_i(T^\vee/\QQ(\mu_{p^\infty}))=\ker\left(\Sel(T^\vee/\QQ(\mu_{p^\infty}))
\rightarrow\frac{H^1(\Qp(\mu_{p^\infty}),T^\vee)}{\ker(\Col_i)^\perp}\right)\]
for $i=1,2$. Here $\Sel(T^\vee/\QQ(\mu_{p^\infty}))$ is the Bloch--Kato Selmer group from \cite{blochkato90}. We shall write $\X(\QQ(\mu_{p^n}))=\Sel(T^\vee/\QQ(\mu_{p^n}))^\vee$ for $n\ge1$.

Let $\X_i$ be the Pontryagin dual of $\Sel_i(T^\vee/\QQ(\mu_{p^\infty}))$. We subsequently assume that for any Dirichlet character $\eta$ that factors through $\Gal(\QQ(\mu_p)/\QQ)$, both $\X_1^\eta$ and $\X_2^\eta$ are $\cO\lb X\rb$-torsion. Note that this is the case if either $k\ge3$ or $a_p=0$ by \cite[Theorem~6.5]{leiloefflerzerbes10}. In particular, $\nabla_n\X_i^\eta$ are defined for $n\gg0$ by Lemma~\ref{lem:kob}(c).

We have the Poitou-Tate exact sequence (see for example \cite[(61)]{leiloefflerzerbes10})
\begin{equation}\label{eq:PT}
\HH^1(T)\rightarrow\image\Col_i \rightarrow \X_i\rightarrow\X_0\rightarrow0,
\end{equation}
where $\X_0$ is $\HH^2(T)$ and can be realized as the Pontryagin dual of the zero Selmer group $\Sel_0(T^\vee/\QQ(\mu_{p^\infty}))$, which is 
defined to be
\[
\ker\left(H^1(\QQ(\mu_{p^\infty}),T^\vee)\rightarrow\prod_v H^1(\QQ(\mu_{p^\infty})_v,T^\vee)\right),
\]
where $v$ runs through all places of $\QQ(\mu_{p^\infty})$. Note that  $\X_0$ is
a torsion $\Lambda$-module by \cite[Theorem~12.4]{kato04} and hence $\nabla_n\X_0^\eta$ is defined for $n\gg0$ by Lemma~\ref{lem:kob}(c). Note that \eqref{eq:PT} gives the short exact sequence
\[
0\rightarrow\frac{\image\Col_i}{(\Col_i(\be))} \rightarrow \X_i\rightarrow\X_0\rightarrow0.
\]
Hence, our assumption that $\X_i^\eta$ be torsion implies that $\Col_i(\be)^\eta\ne0$. In particular, Proposition~\ref{pro:evalXloc} applies.

Recall from \cite[\S5]{leiloefflerzerbes11} that $\image\Col_i^\eta$ is pseudo-isomorphic to $\prod_{m}(X-\chi(\gamma)^m+1)\cO\lb X\rb$, where $m$ runs through some subset of $\{0,1,\ldots, k-2\}$ depending on $i$ and $\eta$. Let us write $\kappa_i(\eta)$ for the cardinality of this subset and write $\kappa(n,\eta)=\kappa_{\tau(n,\eta)}(\eta)$. We have the following generalization of \cite[Proposition~3.11]{sprung12}.

\begin{proposition}
\label{prop:pt}
For $i=1,2$, $\eta$ any Dirichlet character modulo $p$ and $n\gg0$,
\[
\nabla_n\X_i^\eta=\nabla_n(\Lambda/\Col_i(\be))^\eta+\nabla_n\X_0^\eta-e\kappa_i(\eta).
\]
\end{proposition}
\begin{proof}
The following sequence 
\[
0\rightarrow \left(\frac{\image(\Col_i)}{\Col_i(\be)}\right)^\eta\rightarrow\left(\frac{\Lambda}{\Col_i(\be)}\right)^\eta\rightarrow\frac{\cO\lb X\rb}{\prod_{m}(X-\chi(\gamma)^m+1)\cO\lb X\rb} \rightarrow G\rightarrow 0
\]
is exact, where $G$ is some finite subgroup. In particular, $\nabla_nG=0$ for $n\gg0$. We may work out the Kobayashi rank of the second last term using Lemma~\ref{lem:kob}(b). Recall from \cite[Lemma~10.4]{kobayashi03} that Kobayashi ranks respect exact sequences, therefore,
\[
\nabla_n\left(\frac{\image(\Col_i)}{\Col_i(\be)}\right)^\eta+e\kappa_i(\eta)=\nabla_n\left(\frac{\Lambda}{\Col_i(\be)}\right)^\eta.
\]

From \eqref{eq:PT}, we have furthermore the following exact sequence
\[
0\rightarrow \frac{\image(\Col_i)}{\Col_i(\be)}\rightarrow\X_i\rightarrow\X_0\rightarrow0,
\]
which implies that
\[
\nabla_n\left(\frac{\image(\Col_i)}{\Col_i(\be)}\right)^\eta+\nabla_n\X_0^\eta=\nabla_n\X_i^\eta.
\]
Combing the two equations gives our result.
\end{proof}

\begin{remark}
Let $\mu_0^\eta$ be the $\mu$-invariant of $\X_0^\eta$. For $i=1,2$, let $\tilde{\mu}_i^\eta$ be the $\mu$-invariant of $\X_i^\eta$. Then, Proposition~\ref{prop:pt} implies that $\tilde{\mu}_i^\eta=\mu_i^\eta-\mu_0^\eta$. In particular, $\mu_1^\eta-\mu_2^\eta=\tilde{\mu}_1^\eta-\tilde{\mu}_2^\eta$. Therefore, we may replace $\mu_1^\eta$ and $\mu_2^\eta$ by $\tilde{\mu}_1^\eta$ and $\tilde{\mu}_2^\eta$ respectively in Definition~\ref{defn:modesty}. In other words, we may define $\tau(n,\eta)$ using the $\mu$-invariants of the dual Selmer groups $\X_i$, instead of $\Col_i(\be)$.
\end{remark}

\begin{corollary}\label{cor:selmergrowth}
For $n\gg0$, $\nabla\X(\QQ(\mu_{p^{n+1}}))^\eta$ is defined. Furthermore, its value is bounded above by
\[
e q_n^*+\nabla_n\X_{\tau(n,\eta)}^\eta+e\kappa(n,\eta).
\]
\end{corollary}
\begin{proof}
Let $\Y(\QQ(\mu_{p^n}))=\coker (H^1(G_{n,S},T)\rightarrow H^1_{/f}(\Qpn,T))$ and $\X_0(\QQ(\mu_{p^n}))=\Sel_0(T^\vee/\QQ(\mu_{p^n}))^\vee$. As a consequence of the Poitou--Tate exact sequence, we have the short exact sequence
\[
0\rightarrow \Y(\QQ(\mu_{p^n}))\rightarrow\X(\QQ(\mu_{p^n}))\rightarrow\X_0(\QQ(\mu_{p^n}))\rightarrow0
\]
(c.f. \cite[(10.35)]{kobayashi03}). But Proposition~10.6 in \textit{op. cit.} says that
\begin{itemize}
\item $\nabla\Y(\QQ(\mu_{p^{n+1}}))^\eta$ is defined for $n\gg0$ and is equal to $\nabla\Xloc(\QQ(\mu_{p^{n+1}}))^\eta$;
\item $\nabla\X_0(\QQ(\mu_{p^{n+1}}))^\eta=\nabla_n\X_0^\eta$.
\end{itemize} 
  Therefore,
\[
\nabla\X(\QQ(\mu_{p^{n+1}}))^\eta=\nabla\Xloc(\QQ(\mu_{p^{n+1}}))^\eta+\nabla_n\X_0^\eta
\]
and our result follows from Propositions~\ref{pro:evalXloc} and~\ref{prop:pt}.
\end{proof}

\subsection{Bloch--Kato--Shafarevich--Tate groups}

   Let $L$ be a number field. We recall that the Bloch--Kato--Shafarevich--Tate group of $T^\vee$ over $L$ is defined to be
   \begin{equation}
    \label{eq:defnsha}
    \sha(L, T^\vee) = \frac{\Sel(T^\vee/L)}{\Sel(T^\vee/L)_{\div}},
   \end{equation}    
   where $(\star)_{\div}$ denotes the maximal divisible subgroup of $\star$. (See e.g.~\cite[Remark~5.15.2]{blochkato90}). If $f$ corresponds to an elliptic curve $\cE$ and the $p$-primary part of the classical Shafarevich--Tate group $\cE$ is finite, then the two definitions of ($p$-primary) Shafarevich--Tate groups agree.

\begin{proposition}\label{prop:stabilize}
There exists  integers $n_0^\eta,r_\infty^\eta\ge0$ such that 
\[
\corank_{\cO}\Sel(T^\vee/\QQ(\mu_{p^{n+1}}))^\eta= r_\infty^\eta
\]
for all $n\ge n_0^\eta$.
\end{proposition}
\begin{proof}
By Corollary~\ref{cor:selmergrowth}, $\nabla \X(\QQ(\mu_{p^{n+1}}))^\eta$ is defined for $n\gg0$. In particular, the kernel and cokernel of the connecting map
\[
\Sel(T^\vee/\QQ(\mu_{p^{n+1}}))^\vee\rightarrow\Sel(T^\vee/\QQ(\mu_{p^{n}}))^\vee
\]
are finite for $n\gg0$. In particular, $\Sel(T^\vee/\QQ(\mu_{p^{n+1}}))$ and $\Sel(T^\vee/\QQ(\mu_{p^n}))$ must have the same $\Zp$-corank.
\end{proof}

This implies that $\Sel(T^\vee/\QQ(\mu_{p^{n+1}}))^\eta_{\div}\cong(E/\cO)^{\oplus r_\infty^\eta}$ (as $\Zp$-modules) for $n\gg0$. Combined this with \eqref{eq:defnsha}, we obtain the following short exact sequence of $\Zp$-modules
   \[
   0\rightarrow (E/\cO)^{\oplus r_\infty^\eta}\rightarrow\Sel(T^\vee/\QQ(\mu_{p^{n+1}}))^\eta\rightarrow \sha(\QQ(\mu_{p^{n+1}}),T^\vee)^\eta\rightarrow0.
   \]
 Therefore, on taking Pontryagin duals, we deduce that
\[
\nabla\X(\QQ(\mu_{p^{n+1}}))^\eta=r_\infty^\eta+\nabla\sha(\QQ(\mu_{p^{n+1}}),T^\vee)^\eta.
\]
From Corollary~\ref{cor:selmergrowth}, we deduce that
\[
\nabla\sha(\QQ(\mu_{p^{n+1}}),T^\vee)^\eta\le eq_n^*+\nabla_n\X_{\tau(n,\eta)}^\eta+e\kappa(n,\eta)-r_\infty^\eta.
\]
Therefore, we obtain the following theorem on applying Lemma~\ref{lem:finitekob}.
\begin{theorem}\label{thm:final}
Let $\#\sha(\QQ(\mu_{p^{n}}),T^\vee)^\eta= p^{\e_n^\eta}$.
For $n\gg0$, 
\[
\e_{n+1}^\eta-\e_n^\eta\le r\left(eq_n^*+\nabla_n\X_{\tau(n,\eta)}^\eta+e\kappa(n,\eta)-r_\infty^\eta\right),
\]
where $r$ is the integer so that the residue field of $E$ has cardinality $p^r$. 
\end{theorem}

Using Lemma~\ref{lem:kob}, we may rewrite this formula as
\[
\e_{n+1}^\eta-\e_n^\eta\le d\left(q_n^*+\lambda_{\tau(n,\eta)}+(p^{n}-p^{n-1})\frac{\mu_{\tau(n,\eta)}}{e}+\kappa(n,\eta)-\frac{r_\infty^\eta}{e}\right),
\]
where $d=[E:\Qp]$.

\appendix
\section{Growth of Tamagawa numbers over cyclotomic extensions}

We let $T=T_f(j)$ and $\T=T_f(k-1)$ be the representations studied in the main part of the article.  In particular, we assume all the previous hypotheses on $T$ and $\T$ are satisfied throughout. Furthermore, we shall assume that the eigenvalues of $\vp$ on $\Dcris(\T)$ are not integral powers of $p$. For notational simplicity, we shall assume that the coefficient field $E$ is $\Qp$ throughout.

Recall the Perrin-Riou $p$-adic regulator
\[
\cL_\T:\HIw(\Qp(\mu_{p^\infty}),\T)\rightarrow \cH\otimes\Dcris(\T)
\]
defined by $\fM^{-1}\circ(1-\vp)\circ(h^1_\T)^{-1}$, which is the map used to define the Coleman maps in Definition~\ref{defn:Coleman}. We have the following interpolation formula
\begin{proposition}\label{prop:interpolation}
Let $n\ge1$. For any $z\in \HIw(\Qp(\mu_{p^\infty}),\T)$, $i\ge 0$ and a Dirichlet character $\delta$ of conductor $p^n$, we have
\[
\cL_\T(z)(\chi^i\delta)=
\begin{cases}
i!(1-\vp)(1-p^{-1}\vp^{-1})^{-1}\left(\exp^*( z_{0,-i})\right)\cdot t^{-i}e_i&\text{if $n=0$,}\\
\frac{i!p^n}{\tau(\delta)}\vp^n\left(\exp^*(\tilde{e}_\delta\cdot z_{n,-i})\right)\cdot t^{-i}e_i&\text{otherwise,}
\end{cases}
\] 
where $\tau(\delta)$ is the Gauss sum of $\delta$, $z_{n,-i}$ is the projection of $z$ in $H^1(\Qpn,\T(-i))$ and $\tilde{e}_\delta$ is the element $\sum_{\sigma\in\Gal(\Qpn/\Qp)}\delta^{-1}(\sigma)\sigma$.
\end{proposition}
\begin{proof}
This is a slight reformulation of \cite[Theorem~B.5]{loefflerzerbes11} since we have the equation $$\vp(t^{-i}e_i)=p^{-i}\cdot t^{-i}e_i.$$
\end{proof}

\begin{corollary}\label{cor:criterionvanish}
Let $z\in \HIw(\Qp,\T)$. Then, $\cL_\T(z)(\chi^i\delta)=0$ if and only if $\tilde{e}_\delta\cdot z_{n,-i}\in \tilde{e}_\delta\cdot H^1_f(\Qpn,\T(-i))$.
\end{corollary}
\begin{proof}
This is because our assumption on the eigenvalues of $\vp$ implies that $(1-\vp)(1-p^{-1}\vp^{-1})^{-1}$ and $\vp^n$ are both invertible.
\end{proof}

We write $K=\QQ(\mu_{p^n})$ and $\Delta_K=\Gal(K/\QQ)$. For each character $\delta$ on $\Delta_K$, we write $p^{n_\delta}$ for its conductor. Let $K_p$ be the completion of $K$ at the unique place above $p$ (which may be identified with $\Qpn$). We fix a basis $v$ for $\Fil^0\Dcris(T)$ and its dual $v'$ in $\Dcris(T^*(1))/\Fil^0\Dcris(T^*(1))$. We have the definition of the Tamagawa number as defined by Bloch--Kato \cite{blochkato90}:
\[
\Tam(T/K)=[H^1_f(K_p,T):\cO_{K_p}\cdot v]L_p(T,1),
\]
where $L_p(T,1)$ is the Euler factor of the complex $L$-function $L_p(T,1)$ at $p$ and we identify $\cO_{K_v}v$ with its image under the Bloch--Kato exponential map. We may decompose the Tamagawa number into isotypic components, namely
\[
\Tam(T/K)=\prod_{\eta}\Tam(T/K)^\eta,
\]
where the product runs through all the Dirichlet characters modulo $p$ and $\Tam(T/K)^\eta$ is given by
\[
[H^1_f(K_p,T)^\eta:\left(\cO_{K_p}\cdot v\right)^\eta]L_p(T(\eta),1),
\]
which we may identify with $\Tam(T(\eta)/K^\Delta)$.

\begin{lemma}\label{lem:tam}
Let $d_K$ be the discriminant of $K$. Then, we have the formula
\[
\Tam(T/K)=|d_K|_p^{-1}[\cO_{K_p}\cdot v:H^1_{/f}(K_p,T)]L_p(T,1),
\]
where we identify $H^1_{/f}(K_p,T)$ with its image under the Bloch--Kato dual exponential map.
\end{lemma}
\begin{proof}
This follows from the commutative diagram
\[
\begin{CD}
\left({K_p}\otimes\Fil^0\Dcris(T)\right) @.\times@.\left({K_p}\otimes\frac{\Dcris(T^*(1)}{\Fil^0\Dcris(T^*(1))}\right) @>>>{K_p}\\
 @AA\exp^*A @.@VV\exp V @VV\Tr_{K_p/\Qp} V\\
\left(  \Qp\otimes H^1_{/f}(K_p,T)\right) @.\times@.   \left(\Qp\otimes H^1_f(K_p,T^*(1)) \right)@>>> \Qp.
\end{CD}
\]
\end{proof}

 Take $\be$ to be a $\Lambda$-generator of $\HH^1(T)$ as in the main part of the article.  This gives a $\Lambda$-basis $\be\cdot e_{k-j-1}$ of $\HH^1(\T)$.  We shall write $\cL_T(\be)$ for $\Tw_{-k+j+1}\circ\cL_\T(\be\cdot e_{k-j-1})$
and
\[
\tilde{v}_K=\bigotimes_{\delta\in\hat{\Delta}_K}\left(\vp^{n_\delta}(1-\delta(p)\vp)(1-p^{-1}\bar{\delta}(p)\vp^{-1})^{-1}v\right).
\]

\begin{theorem}\label{thm:tamL}
Suppose that $\cL_T(\be)(\delta)\ne0$ for all $\delta\in\hat{\Delta}_K$. Then,
\[
\bigotimes_{\delta\in\hat{\Delta}_K}\cL_T(\be)(\delta)\sim_p 
\frac{\Tam(T/K)}{L_p(T,1)}\prod_{\delta}\left[e_\delta H^1_{/f}(K_p,T):e_\delta\be_K\right]\tilde{v}_K.
\]
Here, we write $a\sim_p b$ if $a$ and $b$ have the same $p$-adic valuation.
\end{theorem}
\begin{proof}
Let $\be_K$ be the projection of $\be$ in $H^1(K_p,T)$. For each character of $\Delta_K$, we write $e_\delta=\sum_{\sigma\in\Delta_K}\delta^{-1}(\sigma)\sigma$ and let $K_\delta$ for the subfield of $K$ defined by the kernel of $\delta$. 
Our assumption means that $e_\delta\cdot \be_K\notin e_\delta \cdot H^1_f(\Qpn,T)$ for all $\delta$ by Corollary~\ref{cor:criterionvanish}. Note that $\sum e_\delta=[K:\QQ]$. On applying Proposition~\ref{prop:interpolation}, we deduce that
\begin{align*}
\bigotimes_{\delta\in\hat{\Delta}_K}\cL_T(\be)(\delta)\sim_p 
&\prod_{\delta}\left[\frac{e_\delta}{[K:\QQ]} \cO[\Delta_K]v:e_\delta \cO[\Delta_K]\frac{p^{n_\delta}}{\tau(\delta)}\exp^*(\be_K)\right]\tilde{v}_K
\\
\sim_p&\prod_{\delta}p^{n_\delta}\left[e_\delta\cO[\Delta_K]\frac{\tau(\delta)}{[K:\QQ]}:e_\delta\cO[\Delta_K]\right]\\&\times\left[e_\delta \cO[\Delta_K]v:e_\delta \cO[\Delta_K]\exp^*(\be_K)\right]\tilde{v}_K.
\end{align*}
Note that the factor $(k-j-1)!$ does not appear because of the Fontaine--Laffaille condition.
Now, \cite[Proposition~1]{gillard79a} tells us that
\[
\left[e_\delta\cO[\Delta_K]\frac{\tau(\delta)}{[K:\QQ]}:e_\delta\cO[\Delta_K]\right]=[K:K_\delta]\left[e_\delta\cO[\Delta_K]\frac{\tau(\delta)}{[K_\delta:\QQ]}:e_\delta\cO[\Delta_K]\right]=1.
\]
Therefore, we deduce from the conductor-discriminant formula that
\[
\bigotimes_{\delta\in\hat{\Delta}_K}\cL_T(\be)(\delta)\sim_p 
|d_K|_p^{-1}\prod_{\delta}\left[e_{\delta}\cO[\Delta_K]v: e_{\delta}\cO[\Delta_K]\exp^*(\be_K)\right]\tilde{v}_K.
\]
Combining this with Lemma~\ref{lem:tam} gives us the result.
\end{proof}

\begin{remark}
There is in fact a similar formula without assuming the non-vanishing of $\Iarith(T)(\delta)$. It would involve Perrin-Riou's $p$-adic height. See \cite[p.180]{perrinriou03}.
\end{remark}

\begin{corollary}
Let $\eta$ be a Dirichlet character modulo $p$.
Under the conditions of Theorem~\ref{thm:tamL}, we have
\[\nabla_n\Xloc^\eta+b_{n+1}^\eta-b_{n}^\eta= q_n^*+\nabla_n(\Zp\lb X\rb/\Col_{\tau(n,\eta)}(\be)^\eta)+p^{n-1}(p-1)n(k-j-1)\]
for $n\gg0$, where $\tau(n,\eta)$ is as defined in Definition~\ref{defn:modesty} and $b_i^\eta$ denotes the $p$-adic valuation of $\Tam(T/\QQ(\mu_{p^i})^{\eta})$ for $i=n,n+1$.
\end{corollary}
\begin{proof}
Let $\Delta_{n+1}$ be the set of Dirichlet characters of conductor $p^{n+1}$ whose $\Delta$-component is $\eta$. Its cardinality is given by $p^{n-1}(p-1)$.
By Theorem~\ref{thm:tamL}, we have
\[
\bigotimes_{\delta\in\Delta_{n+1}}\cL_T(\be)(\delta)\sim_p 
\frac{\Tam(T/\QQ(\mu_{p^{n+1}}))^\eta}{\Tam(T/\QQ(\mu_{p^{n}}))^\eta}\prod_{\delta\in\Delta_{n+1}}\left[e_\delta H^1_{/f}(K_p,T):e_\delta\be_K\right]\vp^{n+1}(v)^{\otimes |\Delta_{n+1}|}.
\]
This gives
\begin{equation}\label{eq:bigrelation}
\bigotimes_{\delta\in\Delta_{n+1}}\vp^{-n-1}\circ\cL_T(\be)(\delta)\sim_p 
\frac{\Tam(T/\QQ(\mu_{p^{n+1}}))^\eta}{\Tam(T/\QQ(\mu_{p^{n}}))^\eta}\prod_{\delta\in\Delta_{n+1}}\left[e_\delta H^1_{/f}(K_p,T):e_\delta\be_K\right]v^{\otimes |\Delta_{n+1}|}.
\end{equation}
Note that $\vp^{-n-1}\circ\Tw_{-k+j+1}=p^{(n+1)(k-j-1)}\Tw_{-k+j+1}\circ\vp^{-n-1}$. The terms appearing on the left-hand side are therefore simply $p^{(n+1)(k-j-1)}\uCol_{T,n+1}(\be)(\delta)$. Therefore, the $p$-adic valuation of the left-hand side of \eqref{eq:bigrelation} is given by
\[
p^{n-1}(p-1)(n+1)(k-j-1)+q_n^*+\ord_{\epsilon_n}\Col_{\tau(n,\eta)}(\be)^{\eta}(\epsilon_n)
\]
thanks to \eqref{eq:valofz}. Hence the result.
\end{proof}

The proof of our Proposition~\ref{pro:evalXloc} implies that the defect of our inequality in Theorem~\ref{thm:final} is in fact given by the length of $\ker\pi^\eta$, where $\pi$ is some projection map. We see here that we may in fact relate this defect to the Tamagawa numbers, namely,
\[
\len_{\Zp}\ker\pi^\eta=b_{n+1}^\eta-b_{n}^\eta-p^{n-1}(p-1)n(k-j-1).
\]

Let $t_n^\eta$ be the integer $s_n^\eta+b_n^\eta$, which is the $p$-adic valuation of $\#\sha(\QQ(\mu_{p^n}),T^\vee)^\eta\times\Tam(T/\QQ(\mu_{p^n}))^\eta$.  The Bloch--Kato conjecture predicts that this quantity should be related to the leading coefficient of the complex $L$ function of $T$ at $1$. Theorem~\ref{thm:final} tells us that we have the equality
\[
t_{n+1}^\eta-t_n^\eta=  q_n^*+\nabla_n\X_{\tau(n,\eta)}^\eta+\kappa(n,\eta)-r_\infty^\eta+p^{n-1}(p-1)n(k-j-1).
\]
for $n\gg0$.
\providecommand{\bysame}{\leavevmode\hbox to3em{\hrulefill}\thinspace}
\providecommand{\MR}{\relax\ifhmode\unskip\space\fi MR }
\providecommand{\MRhref}[2]{%
 \href{http://www.ams.org/mathscinet-getitem?mr=#1}{#2}
}
\providecommand{\href}[2]{#2}

\end{document}